\documentclass[10pt]{amsart} 
\usepackage[utf8]{inputenc}
\usepackage{amsfonts}
\usepackage{graphics}
\usepackage[all, cmtip]{xy}
\usepackage{amsthm,amsfonts,amssymb,amsmath,amsxtra}
\usepackage[all]{xy}
\usepackage{amsmath}
\usepackage{flexisym}
\usepackage{amsfonts}
\usepackage{amsmath}

\usepackage[colorlinks]{hyperref}
\hypersetup{
  citecolor   = blue 
 }

\newtheorem{Theorem}{Theorem}[section]
\newtheorem{Lemma}[Theorem]{Lemma}
\newtheorem{Proposition}[Theorem]{Proposition}
\newtheorem{Main Theorem}[Theorem]{Main Theorem}
\newtheorem{Definition}[Theorem]{Definition}
\theoremstyle{definition}
\newtheorem{Remark}[Theorem]{Remark}

\usepackage{enumitem}
\usepackage{multicol}
\DeclareMathAlphabet\mathbfcal{OMS}{cmsy}{b}{n}

\usepackage[colorlinks]{hyperref}
\makeatletter
\renewcommand*\env@matrix[1][*\c@MaxMatrixCols c]{%
  \hskip -\arraycolsep
  \let\@ifnextchar\new@ifnextchar
  \array{#1}}
\makeatother

\usepackage[colorlinks]{hyperref}
\hypersetup{linkcolor=black}
\usepackage{xcolor}
\usepackage{ wasysym }

\usepackage{enumitem}

\usepackage{xcolor}
\usepackage{ mathdots }

\newcommand{\CO}{\mathcal{O}}

\newcommand{\Q}{\mathbb{Q}}
\newcommand{\Z}{\mathbb{Z}}
\newcommand{\Spec}{{\mathrm {Spec}}}

\newcommand{\even}{ {\rm even}}
\newcommand{\odd}{ {\rm odd}}

\newcommand{\Mloc}{\mathbb{M}^{\rm loc}}

\newcommand{\Mat}{{\rm Mat}}

\newcommand{\quash}[1]{}  
\author[Ioannis Zachos ]{ Ioannis Zachos} 

\address{
Department of Mathematics \\ 
Michigan State University   \\ 
East Lansing, Michigan}
\email{zachosio@msu.edu}

\begin{document}
\date{}
\title[On orthogonal local models of Hodge type]{On orthogonal local models of Hodge type}

\maketitle

\begin{abstract}
We study local models that describe the singularities of Shimura varieties of non-PEL type for orthogonal groups at primes where the level subgroup is given by the stabilizer of a single lattice. In particular, we use the Pappas-Zhu construction and we give explicit equations that describe an open subset around the ``worst" point of  orthogonal local models given by a single lattice. These equations display the affine chart of the local model as a hypersurface in a determinantal scheme. Using this we prove that the special fiber of the local model is reduced and Cohen-Macaulay. 
\end{abstract}
{
  \hypersetup{linkcolor=black}
  \tableofcontents
}

\section{Introduction}
Local models of Shimura varieties are projective flat schemes over the spectrum of a discrete valuation ring. These projective schemes are expected to model the singularities of integral models of Shimura varieties with parahoric level structure. The definition of local model was formalized to some degree by Rapoport and Zink in \cite{RZ}. However, it was soon realized that the Rapoport-Zink construction is not adequate when the group of the Shimura variety is ramified at $p$ and in many cases of orthogonal groups. Indeed, then the corresponding integral models of Shimura varieties are often not flat (\cite{Pappas}). In \cite{PZ}, Pappas and Zhu gave a general group theoretic definition of local models. These local models appear as subschemes of global (“Beilinson-Drinfeld”) affine Grassmannians and are associated to \emph{local model triples}. A LM-triple over a finite extension $F$ of $\mathbb{Q}_p$, for $p\neq 2$, is a triple $(G, \{\mu\}, K)$ consisting of a reductive group $G$ over $F$, a conjugacy class of cocharacters $\{\mu\}$ of $G$ over an algebraic closure of $F$, and a parahoric subgroup $K$ of $G(F)$. We denote by $\Mloc_K(G,\{\mu\})$ the corresponding local model.
 
 In the present paper, we study local models for Shimura varieties for forms of the orthogonal group which are of Hodge but not PEL type. An example of such a Shimura variety is the following: Consider the group $ \mathbf{G}= {\rm GSpin}(\mathbf{V})$, where $\mathbf{V}$ is a (non-degenerate)
orthogonal space of dimension $d \geq 7$ over $\mathbb{Q}$ and the signature of $\mathbf{V}_{\mathbb{R}}$ is  $(d-2, 2)$. Let $\mathbf{D}$ be the space of oriented negative definite planes in $\mathbf{V}_{\mathbb{R}}$. Then the pair $(\mathbf{G},\mathbf{D})$ is a Shimura datum of Hodge type. Further, consider a $ \mathbb{Z}_p$-lattice $\Lambda$ in $V=\mathbf{V} \otimes_{\mathbb{Q}}  \mathbb{Q}_p$,  for which 
\[
p\Lambda^\vee \subset \Lambda\subset \Lambda^\vee, 
\]
where $\Lambda^{\vee}$ is the dual of $\Lambda$ for the corresponding symmetric form. 
We denote by $l$ the distance of the lattice $\Lambda$ to its dual $ \Lambda^{\vee}$, i.e. $l=\text{lg}_{\Z_p}( \Lambda^{\vee}/ \Lambda)$ and we set $l_*={\rm min}(l, d-l)$.  We let $K_1$ be the connected stabilizer of $\Lambda$ in ${\rm SO}({\mathbf V})(\Q_p)$ and let $K$ be the corresponding parahoric subgroup of $\mathbf{G}(\Q_p)$. Now, for a compact open subgroup $\mathbf{K} \subset \mathbf{G}( \mathbb{A}_f )$ of the form $\mathbf{K} = K^p \cdot K_p $ where $K_p = K $ and $K^p$ is sufficiently small, the corresponding Shimura variety is 
$$
{\rm Sh}_{\mathbf{K}}(\mathbf{G},\mathbf{D}) = \mathbf{G}(\mathbb{Q}) \backslash (\mathbf{D} \times \mathbf{G}(\mathbb{A}_f )/ \mathbf{K}).
$$
This complex space has a canonical structure of an algebraic variety over the reflex field $E(\mathbf{G},\mathbf{D})$ (see \cite{Milne}). The work of Kisin and Pappas \cite{KP} gives that orthogonal Shimura varieties as above admit integral models, whose singularities are the ``same" as those of the corresponding PZ local models. Let us mention here that one application of such orthogonal Shimura varieties lies in arithmetic intersection theory. For example, orthogonal Shimura varieties are used in the proof of the averaged Colmez conjecture (see \cite{FGHM} and \cite{AGHP}). In the rest of the paper, we will only consider local models and Shimura varieties  will not appear again.
Note that there is a central extension (see \cite{M.P})
\[
1\to {\mathbb G}_m\to {\rm GSpin}(V)\to {\rm SO}(V)\to 1.
\]
Hence, by \cite[Proposition 2.14]{PRH}, the local model that pertains to the above Shimura variety is $ \Mloc(\Lambda) \! \! = \! \Mloc_{K_1}( {\rm SO}(V),\{\mu\})$ for the LM triple $({\rm SO}(V), \{\mu\}, K_1)$ where $V$, $K_1$, are as above and we take  the minuscule coweight $\mu : \mathbb{G}_m \rightarrow {\rm SO}(V) $ to be given by $\mu(t) = \text{diag} ( t^{-1}, 1,\ldots ,1, t ).$ In fact, we will consider a more general situation in which $\Q_p$ is replaced by a finite field extension $F$ of $\Q_p$ with integers $\CO_F$.
 
 In this paper, we give an explicit description of $\Mloc(\Lambda)$. The difficulty in this task arises from the fact that the construction of PZ local models is inexplicit and group theoretical. In particular, in order to define the PZ local model we have to take the reduced Zariski closure of a certain orbit inside a global affine Grassmannian. We refer the reader to Subsection $2.1$ where the construction of the PZ local models is reviewed. In the case of local models of PEL type one can use the standard representation of the group to quickly represent the local model as a closed subscheme
 of certain linked (classical) Grassmannians (see \cite{PZ}). This is not possible here since the composition $i\cdot \mu$, where 
 $i: {\rm SO}(V)\hookrightarrow {\rm GL}(V)$ is the natural embedding, is not a minuscule coweight and we have to work harder.
Nevertheless, we give explicit equations for an affine chart of the “worst” point of the local model. These equations display this chart as a quadric hypersurface given by the vanishing of a trace in a determinantal scheme of $2 \times 2$ minors. Using this and classical results on determinantal varieties we prove that the special fiber of the affine chart is reduced and Cohen-Macaulay. This implies that the special fiber of the local model is reduced and Cohen-Macaulay. The results in this paper are used in \cite{PaZa} to construct regular integral models of certain related Shimura varieties. Note here that the ``reduced'' result follows from Pappas-Zhu paper \cite{PZ}, which in turn uses Zhu's proof of the Pappas-Rapoport coherence conjecture (see \cite{Z}). We want also to mention the recent work of Haines and Richarz \cite{HR}, where the authors prove in a more general setting that the special fiber of the PZ local models is reduced and Cohen-Macaulay.  
  
Here, we give an independent elementary proof of these properties by using the explicit equations which, as we said above, describe an open subset around the ``worst'' point of our local model. We also calculate the number of the irreducible components of the special fiber of the affine chart. This is equal to the number of irreducible components of the special fiber of the local model. The reason behind these implications lies in the construction of the local model. In particular, as discussed in \cite{PZ} the geometric special fiber of the PZ local model is a union of affine Schubert varieties. Among those there is a unique closed orbit which consists of a single point, the ``worst" point. The one-point stratum lies in the closure of every other stratum. It follows that, if the special fiber of the local model has a certain nice property at the worst point (for example reducedness), then this should hold everywhere (see for example \cite{Go}). Therefore, in the rest of the paper we mainly focus on the affine chart of the worst point. Note that below we denote by $\mathcal{O}$ the ring of integers of $\breve F$, which is the completion of the maximal unramified extension of $F$ in a fixed algebraic closure, and by $k$ the residue field of $\breve F$. 
 
 We now give an overview of the paper: In Section $2$ we review the definition of the PZ local models. In Section $3$ we show how we derive the explicit equations. We describe an affine chart of the worst point of our orthogonal local model in the cases where $(d,l)=(\even,\even)$, $(d,l)= (\odd,\odd)$, $(d,l)= (\even,\odd)$ and $(d,l)= (\odd,\even)$. 
 
Note that when $l$ is even the symmetric form on $V\otimes_F\breve F$ splits and when $l$ is odd the symmetric form on $V\otimes_F\breve F$ is quasi-split but not split. 
 
 The case that $ l_* \leq 1,$ has been considered by Madapusi Pera in \cite{M.P} and also in the joint work of He, Pappas and Rapoport \cite{PRH}. In the last chapter of \cite{PRH}, the authors easily prove that in the case $l_*=0$ the local model is isomorphic to a smooth quadric. With some more work they prove that in the case $l_*=1$ the local model is isomorphic to a quadric which is singular in one point. Here, we assume that $l_* > 1$ (and also $d\geq 5$) and extend these results.

Before stating our main theorems we need some more notation. Thus, let $n= \lfloor d/2 \rfloor $, $ r = \lfloor l/2 \rfloor$ and  $X$ be a $d \times d$ matrix of the form: 
    $$ X =\left[\ 
\begin{matrix}[c|c|c]
E_1 & O_1 & E_2 \\ \hline
B_1 & A & B_2 \\ \hline
E_3 & O_2 & E_4
\end{matrix}\ \right],$$
where  $E_i \in \Mat_{(n-r) \times (n-r)}$, $ O_j \in \Mat_{(n-r) \times l} $, $ B_{\ell} \in \Mat_{l \times (n-r)}$ and $ A \in \Mat_{l \times l}$. We write $\mathcal{O}[X]$, $ \mathcal{O}[B_1|B_2]$ for the polynomial rings over $\mathcal{O}$ with variables the entries of the matrices $X$ and $(B_1|B_2)$ respectively. We also write $\wedge^2 (B_1 \brokenvert B_2)$ for the $2\times2$ minors of $(B_1|B_2)$ and $J_m$ for the unit antidiagonal matrix of size $m$,
$$ J_m := \begin{pmatrix} 
          &  & 1  \\
          & \iddots  & \\
         1 &  &   \\
         
    \end{pmatrix}. $$
In the introduction, we state
our results in the case that $d$ and $l$ have the same parity,  so $d=2n$ and $l=2r$, or $d=2n+1$ and $l=2r+1$. The results when $d$ and $l$ have 
different parity are a bit more involved to state; we refer the reader to Theorem \ref{mthm2} and Section \ref{remaining}.

\begin{Theorem}\label{thm1.0}
Suppose that $d$ and $l$ have the same parity. Then an affine chart of the local model $\Mloc(\Lambda)$ around the worst point is given by $ U_{d, l}=\Spec(R)$ where $R$ is the quotient ring
\[
R=\mathcal{O}[B_1|B_2] /\left(\wedge^2 (B_1 \brokenvert B_2) , Tr( B_2 J_{n-r} B_1^{t}J_{l}) + 2 \pi \right).
\]
\end{Theorem}

Let us mention here that, for $l_* \neq 0$, none of these models are smooth or semi-stable (as follows from \cite[Theorems  5.1, 5.6]{PRH}).

Below we discuss how we derive the equations of the above theorem and then we give the main ingredients of the proof. We write $ S_0, S_1$ for the antidiagonal matrices of size $ d $,
$$S_0 := \begin{pmatrix} 
          &  & 1^{(n-r)}  \\
          & 0^{(l)}  & \\
         1^{(n-r)} &  &   \\
         
    \end{pmatrix},\quad     S_1:=\begin{pmatrix} 
          &  & 0^{(n-r)}  \\
          & 1^{(l)}  & \\
         0^{(n-r)} &  &   \\
         
    \end{pmatrix}
$$
and we define the ideal
$$
I^{\text{naive}}= \left( X^2,\, \wedge^2 X, \,  X^{t} S_0 X - 2 \pi \left(S_0  + \pi S_1 \right)X,\, X^{t} S_1 X  + 2 (S_0  + \pi S_1) X    \right).  
$$ 
Our first step is to show that an affine chart of the PZ local model around the worst point is given as a closed subscheme of the quotient ${\rm M}=\mathcal{O}[X]/I^{\text{naive}}$. We do this in Section 3. 

This $\mathcal{O}$-flat closed subscheme is obtained by adding certain
equations to $I^{\text{naive}}$: Set
\[
I = I^{\text{naive}} + I^{\text{add}}
\]
where 
\[
I^{\text{add}} = \left(Tr(X), \, Tr(A) + 2 \pi, \, B_2 J_{n-r} B_1^{t}- A J_{l} \right) .
\]
We show that $I$ cuts out the $\mathcal{O}$-flat $\Mloc(\Lambda)\cap {\rm M}$, which is an open affine subscheme of $\Mloc(\Lambda)$. By an involved but completely elementary manipulation of the relations describing the ideal $I$ we prove that:

\begin{Theorem}\label{simplifyProp} Suppose that $d$ and $l$ have the same parity. 
The quotient $\mathcal{O}[X]/I$ is isomorphic to $ \mathcal{O}[B_1|B_2]/ \left(\wedge^2 (B_1 \brokenvert B_2) , Tr( B_2 J_{n-r} B_1^{t}J_{l}) + 2 \pi \right).$
 \end{Theorem}
 
It essentially remains to show that $U_{d, l}=\Spec(R)$ is flat over $\CO$. By definition, $U_{d, l}$ is a hypersurface in the determinantal scheme $D=\Spec(\CO[B_1|B_2]/(\wedge^2 (B_1|B_2))$. Since $D$ is Cohen-Macaulay, see \cite[Remark  2.12]{DR}, we can easily deduce that $U_{d, l}$ and $\overline U_{d, l}$ are also Cohen-Macaulay. Flatness of $ U_{d, l}$ follows, see Section \ref{flatness}. Theorem \ref{thm1.0} quickly follows together with the (essentially equivalent) statement:

\begin{Theorem}\label{thm1.1} Suppose that $d$ and $l$ have the same parity. 
An affine chart of the local model $\Mloc(\Lambda)$ around the worst point is given by $ \Spec(\mathcal{O}[X]/I)$, where $I$ is as above.
\end{Theorem}
 
 Using Theorem \ref{thm1.0} and the reducedness of the fibers of PZ local models (see \cite{PZ}) we have that:
\begin{Theorem}\label{intro-s.f.r}
The special fiber of $U_{d,l}$ is reduced.
\end{Theorem}
In Section \ref{special fiber} we give an independent proof of this result by using that the special fiber $\overline U_{d, l}$ is Cohen-Macaulay and generically reduced. 

In the course of proving the reducedness of $\overline U_{d, l}$, we also determine the number of its irreducible components. We find that when $2<l_*$, where $l_*={\rm min}(l, d-l)$ and $l$ is the distance of our lattice to its dual, the special fiber $\overline U_{d, l}$ has two irreducible components. When  $l_*=2$, $\overline U_{d, l}$ has three irreducible components. In fact, we explicitly describe the equations defining the irreducible components of the special fiber.

Similar arguments extend  to the case that $d$ and $l$ have different parity. We give the corresponding hypersurface in a determinantal scheme and the equations of irreducible components of the special fiber
in all cases.

 \smallskip

{\bf Notation:} Let $F$ be a finite extension of $\mathbb{Q}_p$. We assume that $p\neq 2$. Denote with $\CO_F$ the ring of integers of $F$ and let $\pi$ be a uniformizer of $\CO_F$. We denote by $\breve F$ the completion of the maximal unramified extension of $F$  in an algebraic closure $ \bar F$. We denote by $\kappa_F$ the residue field of $F$ and by $k$ the algebraic closure of $\kappa_F$ which is also the residue field of $\breve F$. We also set $\mathcal{O} := \CO_{\breve F}$ for the ring of integers of $\breve F.$
\smallskip

{\bf Acknowledgements:}  I would like to thank G. Pappas for introducing me into this area of mathematics and
 for his patient help during my work on this article. I also thank the referee for their work.

\section{Preliminaries}

\subsection{Local models}
We now recall the construction of the Pappas-Zhu local models. For a more detailed presentation we refer the reader to \cite{P} and \cite{PZ}.

 Let $G$ be a connected reductive group over $F$. Assume that $G$ splits over a tamely ramified extension of $F$. Let $ \{ \mu \}$ be a conjugacy class of a geometric cocharacter $ \mu : \mathbb{G}_{m \bar F} \rightarrow G_{\bar{F}}$ and assume that $\mu$ is minuscule. Define $K$ to be the parahoric subgroup of $G(F)$, which is the connected stabilizer of some point $x$ in the (extended) Bruhat-Tits building $\mathbb{B}(G,F)$ of $G(F)$. Define $E$ to be the extension of $F$ which is the field of definition of the conjugacy class $ \{ \mu \}$ (the reflex field).
 
In \cite{PZ}, the authors construct an affine group scheme $\mathcal{G}$ which is smooth over $\Spec(O_F[u])$ and which, among other properties, satisfies:
\begin{enumerate}
    \item The base change of $\mathcal{G}$ by $\Spec(O_F) \rightarrow \Spec(O_F[u]) = \mathbb{A}^1_{O_F}$ given by $u \rightarrow \pi$ is the Bruhat-Tits group scheme which corresponds to the parahoric subgroup $K$ (see \cite{BTII}).  
    \item The group scheme $\mathcal{G}_{| O_F [u, u^{-1}]}$ is reductive.
\end{enumerate}

Next, they consider the global (“Beilinson-Drinfeld”) affine Grassmannian  
\[
{\rm Aff}_{\mathcal{G},\mathbb{A}^1_{O_F}} \rightarrow \mathbb{A}^1_{O_F}
\]
 given by $\mathcal{G}$, which is an ind-projective ind-scheme. By base changing $u \rightarrow \pi$, they obtain an equivariant isomorphism 
$$ \text{Aff}_G \xrightarrow{\sim} \text{Aff}_{\mathcal{G},\mathbb{A}^1_{O_F}}\times_{\mathbb{A}^1_{O_F}} \Spec(F)$$ 
where $\text{Aff}_G$ is the affine Grassmannian of $G$; this is the ind-projective ind-scheme over $\Spec(F)$ that represents the fpqc sheaf associated to 
$$R \rightarrow G(R((t)))/G(R[[t]]),$$
where $R$ is an $F$-algebra (see also \cite{PR}). (The above isomorphism is induced by $t \rightarrow u-\pi$, see \cite{P}.)

The cocharacter $\mu$ gives an $\bar{F}[t, t^{-1}]$-valued point of $G$ and thus $\mu$ gives an $\bar{F}((t))$-valued point $ \mu(t)$ of $G$. This gives a $ \bar{F}$-point $ [\mu(t)]= \mu(t) G ( \bar{F}[[t]])$ of $\text{Aff}_G $.
Since $\mu$ is minuscule and $ \{\mu \} $ is defined over the reflex field $E$ the orbit $$ G(\bar{F}[[t]])[\mu(t)] \subset \text{Aff}_G (\bar{F}),$$
is equal to the set of $ \bar{F}$-points of a closed subvariety $X_{\mu}$ of $\text{Aff}_{G,E}= \text{Aff}_G \otimes_{F }E .$

\begin{Definition}
Define the local model $\Mloc_K(G, \{\mu\})$ to be the flat projective scheme over $\Spec(O_E)$ given by the reduced Zariski closure of the image of $$X_{\mu} \subset \text{Aff}_G \xrightarrow{\sim} \text{Aff}_{\mathcal{G},\mathbb{A}^1_{O_F}}\times_{\mathbb{A}^1_{O_F}} \Spec(E)$$ in the ind-scheme $\text{Aff}_{\mathcal{G},\mathbb{A}^1_{O_F}}\times_{\mathbb{A}^1_{O_F}} \Spec(O_E)$.
\end{Definition}

The PZ local models have the following property (see \cite[Prop. 2.14]{PRH}).
\begin{Proposition}
 If $F'/F$ is a finite unramified extension, then 
 $$\Mloc_K(G, \{\mu\}) \otimes_{O_E}O_{E'}\xrightarrow{\sim } \Mloc_{K'}(G\otimes_FF',\{\mu\otimes_FF'\}). $$
Note that here the reflex field $E'$ of $(G\otimes_FF', \{\mu\otimes_FF'\})$ is the join of $E$ and $F'$. Also, $K'$ is the parahoric subgroup of $G\otimes_FF'$ with $ K = K' \cap G $.\qed
\end{Proposition}

The above proposition allows us to base change to an unramified extension $F'$ over $F$. This will play a crucial role in the proof of our main theorems.

\subsection{Quadratic forms}
Let $V$ be an $F$-vector space with dimension $d= 2n$ or $2n+1$ equipped with a non-degenerate symmetric $F$-bilinear form $\langle \; , \; \rangle$.
It follows from the classification of quadratic forms over local fields \cite{Ger} that after passing to a sufficiently big unramified extension $F'$ of $F$, the base change of $ ( V , \langle \; , \; \rangle)$ to $F'$ affords a basis as in one of the following cases:

\begin{enumerate}
    \item \textbf{Split form:} there is a basis $f_i$ with the following relations:\\ $\langle f_i , f_{d+1-j} \rangle  = \delta_{ij} $, $\forall i,j \in \{ 1, \ldots , d\}$.
    \smallskip
    
    \item \textbf{Quasi-split form (for} $\mathbf{d=2n}$\textbf{):} there is a basis $f_i $ with the relations: $ \langle f_i, f_{d+1-j}\rangle=\delta_{ij}, $ for $i$, $j\neq n, n+1$, $  \langle f_n, f_n \rangle=\pi$, $\langle f_{n+1}, f_{n+1} \rangle =1$, $\langle f_{n}, f_{n+1} \rangle=0.$
    \smallskip

    \item \textbf{Quasi-split form (for} $\mathbf{d=2n+1}$\textbf{):} there is a basis $f_i $ with the relations: $ \langle f_i, f_{d+1-j} \rangle =\delta_{ij},$ for $i$, $j\neq  n+1$, $\langle f_{n+1}, f_{n+1} \rangle =\pi.$

\end{enumerate}

\section{An Affine Chart}
\subsection{Normal forms of quadric lattices}\label{Normalforms}

Let $V$ be an $F$-vector space with dimension $d= 2n$ or $2n+1$ equipped with a non-degenerate symmetric $F$-bilinear form $\langle \; , \; \rangle$. We assume that $d\geq 5$. For all the cases below we take the minuscule coweight $ \mu : \mathbb{G}_m \rightarrow SO(V) $ to be given by $ \mu(t) = \text{diag} ( t^{-1}, 1,\ldots ,1, t ) $, defined over $F$.

A lattice $\Lambda \subset V$ is called a vertex lattice if $\Lambda \subset \Lambda^{\vee} \subset \pi^{-1} \Lambda $. By $\Lambda^{\vee}$ we denote the dual of $\Lambda $ in $V$: 
$$\Lambda^{\vee} := \{ x\in V | \langle \Lambda , x \rangle  \subset O_F \}.$$
Let $\Lambda$ in $V$ be a vertex lattice. So, $\Lambda \, \subset_l\, \Lambda^{\vee} \, \subset_{l'}\, \pi^{-1} \Lambda$ with $l+l'=d$. Here $l$ (respectively $l'$) is the length $l=\text{lg}( \Lambda^{\vee}/ \Lambda)$ (respectively $l'=\text{lg}(\pi^{-1} \Lambda /\Lambda^{\vee} )$). We assume that $l > 1 $ and $l' > 1 $.

For the following we refer the reader to Rapoport-Zink's book \cite{RZ}, Appendix on Normal forms of lattice chains. More precisely, by \cite[Appendix, Proposition A.21]{RZ}, after an \'etale base change (i.e. an unramified base change) we can find an $O_F$-basis $\{e_i\}$ of $ \Lambda $ with the following property:

For $d=2n$:
\begin{enumerate}
    \item \textbf{Split form:} $\Lambda=\oplus_{i=1}^d O_F\cdot e_i$ 
    with
    \[
    \langle e_i , e_{d+1-j} \rangle\! \!= \delta_{ij}, \ \hbox{\rm for}\, i\not\in [n-r+1, n+r], 
    \]
    \[
    \langle e_i , e_{d+1-j} \rangle\! \!= \pi \delta_{ij}, \ \hbox{\rm for}\, i\in [n-r+1, n+r]. 
    \]
    We have  $\Lambda\, {\subset}_l\, \Lambda^{\vee}$ where $l=2r$.
    
    \item \textbf{Quasi-split form:} $\Lambda=\oplus_{i=1}^d O_F\cdot e_i$ 
    with
    \[
    \langle e_i , e_{d+1-j} \rangle\! \!= \delta_{ij}, \ \hbox{\rm for}\, i \in [1,d]\setminus [n-r,n+r+1 ], 
    \]
    \[
    \langle e_i , e_{d+1-j} \rangle\! \!= \pi \delta_{ij}, \ \hbox{\rm for}\, i\in [n-r,n+r+1 ] \setminus \{n,n+1\},
    \] 
    \[
    \langle e_n ,  e_n \rangle = \pi, \quad \langle e_{n+1} ,  e_{n+1} \rangle = 1, \quad  \langle e_n , e_{n+1} \rangle = 0.
    \]
 We have  $\Lambda\, {\subset}_l\, \Lambda^{\vee}$ where $l=2r+1$.
\end{enumerate}

For $d= 2n+1$:
\begin{enumerate}[resume]
    
      \item \textbf{Split form:} $\Lambda=\oplus_{i=1}^d O_F\cdot e_i$ 
    with
    \[
    \langle e_i , e_{d+1-j} \rangle\! \!= \delta_{ij}, \ \hbox{\rm for}\, i\not\in [n+1-r,n+1+r ] \setminus \{n+1\}, 
    \]
    \[
    \langle e_i , e_{d+1-j} \rangle\! \!= \pi \delta_{ij}, \ \hbox{\rm for}\, i\in [n+1-r,n+1+r ]\setminus \{n+1\}.
    \]
We have  $\Lambda\, {\subset}_l\, \Lambda^{\vee}$ where $l=2r$.
    
     \item \textbf{Quasi-split form:} $\Lambda=\oplus_{i=1}^d O_F\cdot e_i$ 
    with
    \[
    \langle e_i , e_{d+1-j} \rangle\! \!= \delta_{ij}, \ \hbox{\rm for}\, i\not\in [ n+1-r, n+1+r ], 
    \]
    \[
    \langle e_i , e_{d+1-j} \rangle\! \!= \pi \delta_{ij}, \ \hbox{\rm for}\, i\in [ n+1-r, n+1+r ].
    \]
We have  $\Lambda\, {\subset}_l\, \Lambda^{\vee}$ where $l=2r+1$.

\end{enumerate}
 From the above discussion, it follows that we can reduce our problem to the above cases by passing to a sufficiently big unramified extension of $F$. Thus, from now on we will be working over $\breve F.$ Recall that we denote by $\mathcal{O}$ its ring of integers and by $k$ its residue field.

In all cases, we will denote by $S$ the (symmetric) matrix with entries $\langle e_i, e_j \rangle$ where $\{e_i\}$ is the basis above.
We can then write 
\[
S=S_0+\pi S_1
\]
where $S_0$, $S_1$ both have entries only $0$ or $1$. For example, in case (1) we have the anti-diagonal matrices:
\[
S_0 := \begin{pmatrix} 
          &  & 1^{(n-r)}  \\
          & 0^{(2r)}  & \\
         1^{(n-r)} &  &   \\
         
    \end{pmatrix},\quad     S_1:=\begin{pmatrix} 
          &  & 0^{(n-r)}  \\
          & 1^{(2r)}  & \\
         0^{(n-r)} &  &   \\
        \end{pmatrix}.
        \] 
 
\subsection{Lattices over $\mathcal{O}[u]$ and local models}

We can now extend our data to $\mathcal{O}[u,u^{-1}]$. We define $\mathbb{V} = \oplus^{d}_{i=1} \mathcal{O}[u,u^{-1}] \bar{e}_i $ and $\langle \; , \; \rangle \; : \mathbb{V} \times \mathbb{V} \rightarrow  \mathcal{O}[u,u^{-1}] $ a symmetric $\mathcal{O}[u,u^{-1}]$-bilinear form such that the value of $\langle \bar{e}_i , \bar{e}_j \rangle$ is the same as the above for $V$ with the difference that $\pi$ is replaced by $u$. Similarly, we define $\bar{\mu} (t) :  \mathbb{G}_m \rightarrow SO(\mathbb{V}) $ by using the $\{ \bar{e}_i\}$ basis for $\mathbb{V}$.

We also define $\mathbb{L}$ the $\mathcal{O}[u]$-lattice in $\mathbb{V}$ by $\mathbb{L} = \oplus_{i=1}^d \mathcal{O}[u]\cdot \bar{e}_i  $. From the above we see that the base change of $(\mathbb{V}, \mathbb{L},\langle \; , \; \rangle )$ from $\mathcal{O}[u,u^{-1}]$ to $F$ given by $u \mapsto \pi$ is $(V,\Lambda,\langle \; , \; \rangle)$.

Let us now define the local model $\Mloc(\Lambda)= \Mloc_K (SO(V),\{\mu\})$ where $K$ is the parahoric stabilizer of $\Lambda$. We consider the smooth, as in \cite{PZ}, affine group scheme $\underline{\mathcal{G}}$ over $\mathcal{O}[u]$ given by $g \in SO(\mathbb{V}) $ that also preserves $\mathbb{L}$ and $\mathbb{L}^{\vee}$. If we base change by $u \mapsto \pi $ we obtain the Bruhat-Tits group scheme $\mathcal{G}$ of $SO(V) $ which is the stabilizer of the lattice chain $\Lambda \subset \Lambda^{\vee} \subset \pi^{-1} \Lambda$. The corresponding parahoric group scheme is the neutral component $\mathcal{G}^{0}$ of $\mathcal{G}$. The construction of \cite{PZ} produces the group scheme $\underline{\mathcal{G}}^0$ that extends $\mathcal{G}^0$. By construction, there is a group scheme immersion $\underline{\mathcal{G}}^0 \hookrightarrow \underline{\mathcal{G}} $.

The global (“Beilinson-Drinfeld”) affine Grassmannian $\text{Aff}_{\underline{\mathcal{G}},\mathbb{A}^1_{\mathcal{O}}} \rightarrow \Spec(\mathcal{O}[u])$ represents the functor that sends the $\mathcal{O}[u]$-algebra $R$, given by $u\mapsto r$, to the set of projective finitely generated $R[u]$-modules $\mathcal{L}$ of $\mathbb{V}\otimes_{\mathcal{O}} R$ which are locally free such that $(u-r)^{N}\mathbb{L}_R\subset \mathcal{L} \subset (u-r)^{-N}\mathbb{L}_R$ for some $N >>0$ and satisfy 
$$
 \mathcal{L}\,  {\subset\, }_l\, \mathcal{L}^{\vee}\, {\subset\, }_{l'}\, u^{-1} \mathcal{L} 
 $$
  with all graded quotients $R$-locally free and of the indicated rank. Here, we set $\mathbb{L}_R = \mathbb{L} \otimes_{\mathcal{O}} R $.

Consider the $\mathcal{O}$-valued point $[\mathcal{L}(0)]$ given by $\mathcal{L}(0) = \bar{\mu} (u-\pi) \mathbb{L}$. Then, as in the Subsection $2.1$ the local model is the reduced Zariski closure of the orbit $[\mathcal{L}(0)]$ in $ \text{Aff}_{\underline{\mathcal{G}}^0,\mathbb{A}^1_{\mathcal{O}}}\times_{\mathbb{A}^1_{\mathcal{O}}} \Spec(\mathcal{O})$; it inherits an action of the group scheme $ \mathcal{G}^0  = \underline{\mathcal{G}}^0 \otimes_{ \mathcal{O}[u]}\mathcal{O}.$ As in
\cite{PZ}, there is a natural morphism $ \text{Aff}_{\underline{\mathcal{G}}^{0},\mathbb{A}^1_{\mathcal{O}}} \rightarrow \text{Aff}_{\underline{\mathcal{G}},\mathbb{A}^1_{\mathcal{O}}}$
induced by
$\underline{\mathcal{G}}^0 \hookrightarrow \underline{\mathcal{G}} $ which identifies $\Mloc(\Lambda)$ with a closed subscheme of $ \text{Aff}_{\underline{\mathcal{G}},\mathbb{A}^1_{\mathcal{O}}}\times_{\mathbb{A}^1_{\mathcal{O}}} \Spec(\mathcal{O})$ .

By the definition of $\mathcal{L}(0)$ we have
	\[ 
	\xy
	(-21,7)*+{(u-\pi)\mathbb{L} \subset   \mathcal{L}(0)\cap \mathbb{L}\ \ \ \ \ };
	(-4.5,3.5)*+{\rotatebox{-45}{$\,\, \subset\,\,$}};
	(-4.5,10.5)*+{\rotatebox{45}{$\,\, \subset\,\,$}};
	(0,14)*+{\mathbb{L}};
	(0,0)*+{\mathcal{L}(0)};
	(4,3.5)*+{\rotatebox{45}{$\,\, \subset\,\,$}};
	(4,10.5)*+{\rotatebox{-45}{$\,\, \subset\,\,$}};
	(23.5,7)*+{\ \ \ \ \ \mathbb{L}+\mathcal{L}(0) \subset   (u-\pi)^{-1}\mathbb{L},};
	\endxy    
	\]
where the quotients along all slanted inclusions are $\mathcal{O}$-free of rank $1$ (for more details see proof of Proposition \ref{L0}). Let us define M to be the subfunctor of $\text{Aff}_{\underline{\mathcal{G}},\mathbb{A}^1_{\mathcal{O}}}\times_{\mathbb{A}^1_{\mathcal{O}}} \Spec(\mathcal{O})$ that parametrizes all $\mathcal{L}$ such that 
$$ (u-\pi)\mathbb{L }\subset \mathcal{L} \subset (u-\pi)^{-1}\mathbb{L} .$$ 
Then M is represented by a closed subscheme of $\text{Aff}_{\underline{\mathcal{G}},\mathbb{A}^1_{\mathcal{O}}}\times_{\mathbb{A}^1_{\mathcal{O}}} \Spec(\mathcal{O})$ which contains $[\mathcal{L}(0)]$. In that way, $\Mloc(\Lambda)$ is a closed subscheme of M and $\Mloc(\Lambda)$ is the reduced Zariski closure of its generic fiber in M. As in \cite[Proposition 12.7]{PRH}, the elements of $\Mloc(\Lambda)$ have the following properties:

\begin{Proposition}\label{L0}
If $\mathcal{L} \in \Mloc (\Lambda)(R) $, for an $\mathcal{O}$-algebra $R$, then:
\begin{enumerate}
    \item $\mathcal{L}$ is $u$-stable,
    \item $ \mathcal{L}\, {\subset\, }_l\, \mathcal{L}^{\vee} $, and 
    \item \[ 
	\xy
	(-21,7)*+{(u-\pi) \mathbb{L}_R  \subset   \mathcal{L}\cap \mathbb{L}_R \ \ \ \ \ };
	(-4.5,3.5)*+{\rotatebox{-45}{$\,\, \subset\,\,$}};
	(-4.5,10.5)*+{\rotatebox{45}{$\,\, \subset\,\,$}};
	(0,14)*+{\mathbb{L}_R };
	(0,0)*+{\mathcal{L}};
	(4,3.5)*+{\rotatebox{45}{$\,\, \subset\,\,$}};
	(4,10.5)*+{\rotatebox{-45}{$\,\, \subset\,\,$}};
	(23.5,7)*+{\ \ \ \ \ \mathbb{L}_R +\mathcal{L} \subset   (u-\pi)^{-1}\mathbb{L}_R ,};
	\endxy    
	\]
	where the quotients arising from all slanted inclusions are generated as $R$-modules by one element (we say that they have rank $\leq 1$).
\end{enumerate}
\end{Proposition} 
\begin{proof}
The first two conditions follow directly from the definition of the local model. By the definition of $ \mathcal{L}(0) $ we have $\mathcal{L}(0) = \bar{\mu} (u-\pi) \mathbb{L}$ where $ \bar{\mu}(u-\pi) = \text{diag} ( (u-\pi)^{-1}, 1,\ldots ,1, u-\pi)  $. We can easily see that $(3)$ is true for $ \mathcal{L}(0) $. Since condition $(3)$ is closed and $\mathcal{G}$-equivariant  
it also holds for $\mathcal{L}$ and the proposition follows.
\end{proof}

Define $\mathcal{F}^{\prime}$ to be the image of $\mathcal{L}$  by the map
$$
(u-\pi)^{-1}\mathbb{L}_R /(u-\pi)\mathbb{L}_R  \xrightarrow{u-\pi}\mathbb{L}_R /(u-\pi)^2 \mathbb{L}_R .
$$ 
Define the symmetric bilinear form: 
$$ 
\langle \; , \; \rangle \textprime :  \mathbb{L}/(u-\pi)^2 \mathbb{L} \times  \mathbb{L}/(u-\pi)^2 \mathbb{L} \rightarrow \mathcal{O}[u]/(u-\pi)^2\mathcal{O}[u],
 $$
by 
\[
 \langle \; , \; \rangle\textprime = \langle \; , \; \rangle \text{mod}  (u-\pi)^2.
 \]
  Notice, that condition $(2)$ above means that $\langle \mathcal{L}, \mathcal{L} \rangle \in R[u]$ under the $R$-base change of the bilinear form $\langle \; , \; \rangle$. Thus, $\mathcal{F}^{\prime}$ is isotropic for $\langle \; , \; \rangle\textprime $ on $\mathbb{L}_R/(u-\pi)^2 \mathbb{L}_R \times  \mathbb{L}_R/(u-\pi)^2 \mathbb{L}_R$, i.e. $ \langle  \mathcal{F}^{\prime}, \mathcal{F}^{\prime} \rangle\textprime = 0  $. We also observe that $ {\rm rank}(u-\pi) \leq 1 $ where $ u-\pi : \mathcal{F}^{\prime} \rightarrow \mathcal{F}^{\prime} $. That follows from condition $(3)$ and the fact that $ (u-\pi)^2 \mathbb{L}_R = 0 $ in $\mathbb{L}_R/(u-\pi)^2 \mathbb{L}_R$.

\subsection{The affine chart}
For the sake of simplicity we fix $d = 2n$ and $l = 2r$. We get similar results for all the other cases.

For any $\mathcal{O}$-algebra $R$, let us  consider the $R$-submodule:
$$
\mathcal{F}  = \{ (u-\pi)v + Xv\, |\, v \in R^d \} \subset (u-\pi) R^d \oplus R^d \cong \mathbb{L}_R /(u-\pi)^2 \mathbb{L}_R 
$$
with $X \in \Mat_{d\times d} (R)$. 

We ask that $\mathcal{F}$  satisfies the following three conditions:
\begin{enumerate}
   \item $\mathbf{u}$\textbf{-stable:} It suffices to be $(u-\pi)$-stable. Let $(u-\pi)v + Xv \in \mathcal{F}$. Then there exists $w \in  R^d$ such that $ (u-\pi)^2v + (u-\pi) Xv = (u-\pi)w + Xw $. This gives $ X u v - X \pi v = uw - \pi w+ X w $ and so: 
    $$ w= Xv, $$ 
    $$ -\pi Xv = -\pi w + Xw .$$ 
    By substituting the former equation to the latter, we have $ X^2 v =0  $. Because this is correct for every $v$, we have $ X^2 =0$. Observe that $X$ is the matrix giving multiplication by $(u-\pi)$ on $\mathcal{F}$.
  
    \item \textbf{Isotropic:} Let $(u-\pi)v + Xv \in \mathcal{F}$. We want  
    $$ \langle (u-\pi)v + Xv , (u-\pi)v + Xv \rangle\textprime = 0 $$
    and recall that $\langle \; , \; \rangle\textprime = \langle \; , \; \rangle \text{mod}  (u-\pi)^2$. By simplifying the above equation we have 
    $$ -2 ( u - \pi )\langle v , Xv \rangle\textprime  = \langle Xv , Xv \rangle\textprime .$$  
    The above relation holds for any $v$ and so we get: 
    $$ -2 ( u - \pi ) (S_0 + u S_1) X  = X ^ {t}(S_0 + u S_1) X$$
    where $S_0$, $S_1$ are the matrices with $S=S_0+\pi S_1=(\langle e_i, e_j\rangle)_{i, j}$ as in \S \ref{Normalforms}.
    By simplifying the above relation we have:
    $$ \quad \quad \quad  2 \pi S_0 X + 2 \pi^2 S_1 X + u ( -2 \pi S_1 X - 2 S_0 X) = X^{t} S_0 X+ u (X^{t} S_1 X)$$ which amounts to 
   $$ X^{t} S_0 X = 2 \pi (S_0 X + \pi S_1 X) \ \hbox{\rm and}\ X^{t} S_1 X  = -2 (S_0 X + \pi S_1 X ). $$
    \item \textbf{rank}$\boldsymbol{( u-\pi}$ $\! \!  \mathbf{|} $ $ \mathbfcal{F'}$ $\boldsymbol{ \!  ) \leq 1 :}$ By the above, this translates to $ \wedge^2 X = 0 $.

\end{enumerate}

Let $\mathcal{U}^{\text{naive}}$ be the corresponding scheme of $\mathcal{F}$ defined by the $d \times d$ matrices $X$ which satisfy the following relations:
$$ X^2 = 0 , \quad X^{t} S_0 X - 2 \pi (S_0 X + \pi S_1 X)=0,
$$
$$  \wedge^2 X = 0 , \quad X^{t} S_1 X  + 2 (S_0 X + \pi S_1 X ) =0.
$$ 
We denote by $ I^{\text{naive}}$ the ideal generated by the entries of the above relations. 

The conditions (1)-(3) are necessary but not always sufficient for $\mathcal{L}$ to correspond to an $R$-valued point of $\Mloc (\Lambda)$. Let $\mathcal{U} = \Mloc(\Lambda)\cap \mathcal{U}^{\text{naive}}$.

\begin{Lemma}\label{GenericFibers}
The generic fibers of $\mathcal{U}$ and $\mathcal{U}^{\text{naive}} $ are not equal.
\end{Lemma}
\begin{proof}
The generic fiber of $\mathcal{U}^{\text{naive}}$ contains the additional $\breve F$-point $\mathcal{L} = \mathbb{L}_{\breve F}$ which is not in the orbit $[\mathcal{L}(0)]$ of $\bar{\mu}$ in the affine Grassmannian $\text{Aff}_G$.  See also \cite[Remark 4.2.3]{PaZa} and the proof of \cite[Theorem 5.1.6]{PaZa} for more details.
\end{proof}

Also calculations, using Macaulay 2, in dimensions $d=6,7,8$ show that $\mathcal{U}^{\text{naive}}$ has non-reduced special fiber. 
 
Our goal is to calculate the $\mathcal{O}$-flat closed subscheme $\mathcal{U}$ of $\mathcal{U}^{\text{naive}}$ by adding some explicit relations
in the ideal $ I^{\text{naive}}$. The resulting $\mathcal{U}$ is an open subscheme of $\Mloc(\Lambda)$.

Observe that the point $\mathbb{L}$ is fixed by the action of the group scheme $\mathcal{G}^{0}$ and so its our worst point. Thus, $ \mathcal{U}$ is an open neighborhood around the worst point $\mathbb{L}$. Then these additional
relations, together with $I^{\rm naive}$, give explicit equations that describe an open subset around the worst point of our local model $ \Mloc(\Lambda)$.

We introduce some notation that will help us defining those relations. We first rewrite our matrix $X := (x_{ij})_{1 \leq i,j \leq d} $ as follows:

$$ X =\left[\ 
\begin{matrix}[c|c|c]
E_1 & O_1 & E_2 \\ \hline
B_1 & A & B_2 \\ \hline
E_3 & O_2 & E_4
\end{matrix}\ \right]$$
where $E_i \in \Mat_{(n-r) \times (n-r)}$, $ O_j \in \Mat_{(n-r) \times l} $, $ B_{\ell} \in \Mat_{l \times (n-r)}$ and $ A \in \Mat_{l \times l}$. We denote by $\mathcal{O}[X]$ the polynomial ring over $\mathcal{O},$ with variables the entries of the matrix $X.$ We also write $J_m$ for the unit antidiagonal matrix of size $m$,
$$ J_m := \begin{pmatrix} 
          &  & 1  \\
          & \iddots  & \\
         1 &  &   \\
         
    \end{pmatrix} .$$
We will show that by adding the following relations:
$$ Tr(X)=0 ,\quad Tr(A) + 2 \pi =0 , \quad B_2 J_{n-r} B_1^{t}- A J_{l} =0,   $$
we get the desired $\mathcal{O}$-flat scheme $\mathcal{U}$ in the cases where $(d,l)=$(even,even) and $(d,l)=$ (odd,odd). By adding similar relations we get the corresponding result in cases where $(d,l)=$(even,odd) and $(d,l)=$(odd,even). Next, we state the main theorems of this paper.

\begin{Theorem}\label{mthm1}
Suppose that $d$ and $l$ have the same parity so $d=2n$ and $l=2r$, or $d=2n+1$ and $l=2r+1$. Then an affine chart of the local model $\Mloc(\Lambda)$ around the worst point $  \mathcal{L} = \mathbb{L}_{\bar{\mathbb{F}}_p}$ is given by $ U_{d, l}=\Spec(\mathcal{O}[X]/I)$, which is defined by the quotient of the polynomial ring $\mathcal{O}[X]=\CO[(x_{i, j})_{1\leq i, j\leq d}]$ by the ideal
$$ I = I^{\text{naive}} + I^{\text{add}} $$
where 
$$I^{\text{add}} = \left(Tr(X), \, Tr(A) + 2 \pi  , \, B_2 J_{n-r} B_1^{t}- A J_{l} \right) .$$

\end{Theorem}
 
Next, we state the theorems for the cases where $d$ and $l$ have different parity.  In each case we consider $d\times d$ matrices $X$. In order to define the submatrices ($E_i$, $O_j$, $B_{\ell}$, $A$) giving the block decomposition of $X$ we set: 
 $$r' =
\left\{
	\begin{array}{ll}
		r  & \mbox{if } l =2r  \\
		r+1 & \mbox{if } l=2r+1.
	\end{array}
\right.$$

Then we write the matrix $X$ as before, with blocks $ E_i  \! \in \! \Mat_{(n-r') \times (n-r')}$, $ O_j \! \in \! \Mat_{(n-r') \times (l+1)} $, $ A \! \in \! \Mat_{(l+1) \times (l+1)}$ and $ B_{\ell} \! \in \! \Mat_{(l+1) \times (n-r')}$. We denote by $A' $ the matrix $A$ without the $(n+1)$-row and $(n+1)$-column. Lastly, we denote by $ Q $ the $(r'+1 )$-column of $A$, we set $ H = \text{diag}(1^{(r')},0,1^{(r)})$ and $$J'_{l+1} =
\left\{
	\begin{array}{ll}
		J_{l+1}  & \mbox{if } l =2r  \\
		\text{diag}(0^{(r)},1^{(2)},0^{(r)}) + \text{antidiag}(1^{(r)},0^{(2)},1^{(r)})& \mbox{if } l=2r+1.
	\end{array}
\right.$$

\begin{Theorem}\label{mthm2}
 Suppose that $d$ and $l$ have opposite parity, so $d= 2n+1$ and $l=2r $
 or $d=2n$ and $l=2r+1$. An affine chart of the local model $\Mloc(\Lambda)$ around the worst point $ \mathcal{L} = \mathbb{L}_{\bar{\mathbb{F}}_p}$ is given by $  U_{d, l}= \Spec(\mathcal{O}[X]/I)$, which is defined by the quotient of the polynomial ring $\mathcal{O}[X]$ by the ideal
$$ I = I^{\text{naive}} + I^{\text{add}} $$
where 
\[
I^{\text{add}} = \big(Tr(X),   Tr( A') + 2 \pi  ,  2 H  B_2 J_{n-r'} B_1^{t} H  + H  QQ^{t} H  - 2 H  A J'_{l+1} H  \big) .
\]
    \end{Theorem}  
    
In sections \ref{sect4}-\ref{flatness} we carry out the proof of Theorem \ref{mthm1} for the case $(d,l)=(\even,\even)$. The proof of the remaining cases of parity for $d$ and $l$ is given in Section \ref{remaining}. 

Using Theorems \ref{mthm1}, \ref{mthm2} and the fact that PZ local models have reduced special fiber, see  \cite[Theorem 9.1]{PZ}, we obtain:

\begin{Theorem}\label{spec.fib.red.}
The special fiber of $U_{d,l}$ is reduced.
\end{Theorem}

Note that in the above theorem we do not specify the parity of $d$ and $l$. In Section \ref{special fiber} we give an independent proof of this theorem, for the case $(d,l)=(\even,\even)$, by showing that the special fiber is Cohen-Macaulay and generically reduced. A similar argument works for the rest of the cases of parity for $d$ and $l$.

\section{Reduction of Relations}\label{sect4}

In all of Section \ref{sect4}, we assume $d=2n$ and $l=2r$. Our goal in this
section is to prove the simplification of equations given by Theorem \ref{simpler}. 
(This corresponds to Theorem \ref{simplifyProp} of the introduction.)

We are working over the polynomial ring $ S := \mathcal{O}[( x_{i,j})_{1 \leq i,j \leq d}]$. We also set 
\[
S'':=  \mathcal{O}[(x_{t,s})_{ t \in Z, s \in Z^{c} }]
\]
 where $Z : = \{ n-(r-1), \ldots,n,n+1,  \ldots ,d-n+r\}  $ and $ Z^{c}: = \{ 1,2,3, \ldots, d\} \setminus Z $. 
 
 Recall that
$$
I = \big( X^2  , \,\,\, \wedge^2 X  ,  \,\,\, Tr(X), \,\,\, Tr(A) + 2 \pi , \,\,\, B_2 J_{n-r} B_1^{t}- A J_{2r}, $$ $$ X^{t} S_0 X - 2 \pi (S_0 X + \pi S_1 X), \,\,\,  X^{t} S_1 X  +2 (S_0 X + \pi S_1 X  )\big).
$$
We set
$$ I'' = \left(\wedge^2 (B_1 \brokenvert B_2) , Tr(B_2 J_{n-r} B_1^{t}J_{2r}) + 2 \pi \right) $$
where 
$$ \wedge^2 (B_1 \brokenvert B_2) := \Big( x_{i,j}x_{t,s} - x_{i,s}x_{t,j} \Big) _{\substack{i,t \in Z,\, j,s \in Z^{c}}}.  
$$

\begin{Theorem}\label{simpler}
There is an $\CO$-algebra isomorphism $S/I\cong S''/I''$.
\end{Theorem}

\begin{proof}
We define the ideal:
$$ I' = \left(  \wedge^2 X,\,   Tr(X),\,  Tr(A) + 2 \pi,\,  B_2 J_{n-r} B_1^{t}- A J_{2r},\,    X^{t} S_1 X  +2 (S_0 X + \pi S_1 X   )\right). $$
The proof will be done in two steps:
\begin{enumerate}
  
  \item Show $I=I'$. 
\smallskip

 \item Show $ S/I'\cong S''/I''$.
\end{enumerate}

\subsection{Step 1}
Our first reduction is to prove that $ I' = I,$ which will be given in Proposition \ref{I'}. To do that, we are going to show that the entries of $X^2$, $X^{t} S_0 X - 2 \pi (S_0 X + \pi S_1 X)$ are in the ideal $I'$. Proposition \ref{I'} will easily follow. The first relation is more straightforward:
\begin{Lemma}\label{sqrt}
The entries of $X^2$ are in the ideal $ I' $.
\end{Lemma}
\begin{proof}
Let $  (z_{i,j})_{1 \leq i,j \leq d} :=  X^2  $, where $ z_{i,j} =\displaystyle \sum_{a=1}^{d} x_{i,a}x_{a,j}$. Now, set 
$$t_{i,j} := x_{i,j}Tr(X) \in I' .$$
Notice also that 
$$s^{i,j}_a := x_{i,a}x_{a,j} - x_{i,j}x_{a,a}  \in I' $$
from the minors relations. Therefore
$$ t_{i,j}+ \displaystyle \sum_{a=1}^{d} s^{i,j}_a  = z_{i,j} \in I'.$$
\end{proof}
 
We have to work harder in order to show that the entries of $ X^{t} S_0 X - 2 \pi (S_0 X + \pi S_1 X) $ are in the ideal $I'.$ The first step is as follows. By a simple direct calculation the relation $ X^{t} S_1 X  +2 S_0 X + 2 \pi S_1 X = 0$ implies that:
\begin{eqnarray}
     E_1 &=& -\frac{1}{2} J_{n-r} B_2^{t}J_{2r} B_1,\\
     E_2 &=& -\frac{1}{2} J_{n-r} B_2^{t}J_{2r} B_2 ,\\
     E_3 &=& -\frac{1}{2} J_{n-r} B_1^{t}J_{2r} B_1 ,\\
     E_4 &=& -\frac{1}{2} J_{n-r} B_1^{t}J_{2r} B_2 ,\\
     O_1 &=& -\frac{1}{2} J_{n-r} B_2^{t}J_{2r} A  ,\\
     O_2 &=& -\frac{1}{2} J_{n-r} B_1^{t}J_{2r} A .
\end{eqnarray}
Therefore, all the entries from $E_i$ for $ i \in \{1,2,3,4\}$ and $O_1,O_2$ can be expressed in terms of the entries of $ B_1$, $B_2$. The second step is the following lemma.
\begin{Lemma}\label{sym}
Assume that all the $ 2\times2 $ minors of the matrix $X$ are $0.$ Then, the matrix $B_1 J_{n-r} B_2^{t} $ is symmetric.
   \end{Lemma}
   \begin{proof}
   Set $(\theta_{ij})_{1 \leq i,j \leq 2r} := B_1 J_{n-r} B_2^{t}.$ By direct calculations we  find  
   $$ \theta_{ij} = \displaystyle \sum_{t=1}^{n-r} x_{n-r+i,n-r-t+1}x_{n-r+j,n+r+t}.$$
   So,$$ 
   \theta_{ji} = \displaystyle \sum_{t=1}^{n-r} x_{n-r+j,n-r-t+1}x_{n-r+i,n+r+t}.$$
   From the minor relations we have that $ x_{i,j} x_{t,s} = x_{i,s}x_{t,j}.$ Using this and the description of the $\theta_{ij},\theta_{ji}$ we can easily see that $\theta_{ij}=\theta_{ji}.$
 \end{proof}
A useful observation, which will be used in the following lemma, is that the condition $\wedge^2 X=0$ together with the fact that the blocks $B_1$, $A$, and $B_2$, 
all share the same rows of $X$, easily give
\begin{equation}\label{TrAB}
AB_1={\rm Tr}(A)B_1, \quad AB_2={\rm Tr}(A)B_2.
\end{equation} 
We are now ready to show:   
\begin{Lemma}\label{eq1}

The entries of $X^{t} S_0 X - 2\pi(S_0 X+ \pi S_1 X)$ are in the ideal $I'$.

\end{Lemma}
\begin{proof}
Using the block form of the matrix $X$ and the relation $ X^{t} S_1 X  +2 (S_0 X + \pi S_1 X   ) = 0 $ modulo $I'$, it suffices to prove that: 
\begin{enumerate}[label=(\roman*)]
    \item $ E^{t}_1 J_{n-r}E_3 + E^{t}_3 J_{n-r}E_1 -2 \pi J_{n-r}E_3= 0$,
    \item $ E^{t}_2 J_{n-r}E_4 + E^{t}_4 J_{n-r}E_2 -2 \pi J_{n-r}E_2= 0$,
    \item $ E^{t}_1 J_{n-r}E_4 + E^{t}_3 J_{n-r}E_2 -2 \pi J_{n-r}E_4= 0$,
    \item $ E^{t}_2 J_{n-r}E_3 + E^{t}_4 J_{n-r}E_1 -2 \pi J_{n-r}E_1= 0$,
    \item $ O^{t}_1 J_{n-r}E_3 + O^{t}_2 J_{n-r}E_1 -2 \pi^2 J_{2r}B_1= 0$,
    \item $ O^{t}_1 J_{n-r}E_4 + O^{t}_2 J_{n-r}E_2 -2 \pi^2 J_{2r}B_2= 0$,
    \item $ O^{t}_1 J_{n-r}O_2 + O^{t}_2 J_{n-r}O_1 -2 \pi^2 J_{2r}A = 0$
\end{enumerate}
in the quotient ring $  S/I'. $
We prove the first relation (i) and with the same arguments we can prove the relations (ii)-(iv). Below we use the relations $(1)$ and $(3)$ for $E_1, E_3$ from above, the relations $B_2 J_{n-r} B_1^{t}= A J_{2r} $, $ A B_1 = Tr(A) B_1$ and Lemma \ref{sym}.
\begin{eqnarray*} &&E^{t}_1 J_{n-r}E_3 + E^{t}_3 J_{n-r}E_1 -2 \pi J_{n-r}E_3  \\ 
 &=& \frac{1}{4} B^{t}_1 J_{2r}B_2 J_{n-r} B^{t}_1 J_{2r}B_1 + \frac{1}{4} B^{t}_1 J_{2r}B_1 J_{n-r} B^{t}_2 J_{2r}B_1+ \pi B^{t}_1 J_{2r}B_1\\
&=& \frac{1}{2} B^{t}_1 J_{2r}B_2 J_{n-r} B^{t}_1 J_{2r}B_1 + \pi B^{t}_1 J_{2r}B_1\\
&=&  \frac{1}{2} B^{t}_1 J_{2r}A B_1 + \pi B^{t}_1 J_{2r}B_1 = \frac{1}{2} Tr(A) B^{t}_1 J_{2r} B_1 + \pi B^{t}_1 J_{2r}B_1 = 0. 
\end{eqnarray*} 
The last equality holds because $ Tr(A) + 2\pi = 0.$

Next, we prove the relation (v). The relations (vi), (vii) can be proved using similar arguments. We use the relations $(1)$, $(3)$, $(5)$ and $(6)$ from above to express $E_1$, $E_3$, $O_1$, $O_2$ in terms of $B_1$ and $B_2.$ We use Lemma \ref{sym} and the relations $B_2 J_{n-r} B_1^{t}= A J_{2r} $ and $ A B_1 = Tr(A) B_1$. 
\begin{eqnarray*} && O^{t}_1 J_{n-r}E_3 + O^{t}_2 J_{n-r}E_1 -2 \pi^2 J_{2r}B_1  \\ 
 &=& \frac{1}{4} A^{t} J_{2r} B_2 J_{n-r}B^{t}_1 J_{2r} B_1 +  \frac{1}{4} A^{t} J_{2r} B_1 J_{n-r}B^{t}_2 J_{2r} B_1 -2 \pi^2 J_{2r}B_1 \\  
 &=& \frac{1}{2} A^{t} J_{2r} B_2 J_{n-r}B^{t}_1 J_{2r} B_1  -2 \pi^2 J_{2r}B_1 = \frac{1}{2} A^{t} J_{2r} A B_1  -2 \pi^2 J_{2r}B_1 \\
 &=& - \pi J_{2r} A B_1 -2 \pi^2 J_{2r}B_1  = -\pi( Tr(A) J_{2r}  B_1 + 2 \pi J_{2r}B_1 )  = 0 .
\end{eqnarray*}  
\end{proof}
\begin{Proposition}\label{I'}
We have $ I'= I.$
\end{Proposition}
\begin{proof}
From Lemma \ref{sqrt} and Lemma \ref{eq1} we get the desired result.
\end{proof}
\subsection{Step 2.}
The goal of this subsection is to prove that $S/I'$ is isomorphic to $ S''/I'' $. 
Recall  
$$ I' = \left(  \wedge^2 X,\,   Tr(X),\,  Tr(A) + 2 \pi,\,  B_2 J_{n-r} B_1^{t}- A J_{2r},\,    X^{t} S_1 X  +2 (S_0 X + \pi S_1 X   )\right). $$  
We first simplify and reduce the number of generators of $I'.$  The desired isomorphism will then follow.
\begin{Lemma}\label{Rmrk2}
The trace $Tr(X)$ belongs to the ideal
$$\left(  \wedge^2 X  ,  \,\,\,Tr(A) + 2 \pi ,  \,\,\, B_2 J_{n-r} B_1^{t}- A J_{2r},   \,\,\,  X^{t} S_1 X  +2 (S_0 X + \pi S_1 X   )\right).$$ 
 \end{Lemma}
 \begin{proof}
  We first write: $$Tr(X) = Tr(E_1) + Tr(E_4) + Tr(A).$$ 
  By the relations $(1)$, $(4)$ from Subsection $4.1$ we get that the entries of $  E_1 + \frac{1}{2}J_{n-r}B^{t}_2 J_{2r}B_1 $ and $ E_4 + \frac{1}{2}J_{n-r}B^{t}_1 J_{2r}B_2 $, belong to the ideal  
  $$\left(  \wedge^2 X  , \quad Tr(A) + 2 \pi , \quad B_2 J_{n-r} B_1^{t}- A J_{2r},  \quad  X^{t} S_1 X  +2 (S_0 X + \pi S_1 X   )\right). $$
  Also, the element 
  $$ Tr( J_{n-r}B^{t}_1 J_{2r}B_2 ) - Tr(A)$$
  belongs to the above ideal.
  Thus, 
  \begin{eqnarray*}  
 Tr(X)  &=& Tr(E_1) + Tr(E_4) + Tr(A) \\ 
   &=& Tr(E_1 + \frac{1}{2}J_{n-r}B^{t}_2 J_{2r}B_1  ) + Tr(E_4+ \frac{1}{2}J_{n-r}B^{t}_1 J_{2r}B_2) + Tr(A) \\
    && - \frac{1}{2}Tr( J_{n-r}B^{t}_1 J_{2r}B_2)- \frac{1}{2}Tr(  J_{n-r}B^{t}_2 J_{2r}B_1) \\
   &=& Tr(E_1 + \frac{1}{2}J_{n-r}B^{t}_2 J_{2r}B_1  ) + Tr(E_4+ \frac{1}{2}J_{n-r}B^{t}_1 J_{2r}B_2),
    \end{eqnarray*}   
     belongs to the above ideal, as desired.\end{proof}
     
From the above lemma we obtain
$$I'=\left(  \wedge^2 X  ,  \,\,\,Tr(A) + 2 \pi ,  \,\,\, B_2 J_{n-r} B_1^{t}- A J_{2r},   \,\,\,  X^{t} S_1 X  +2 (S_0 X + \pi S_1 X   )\right). $$
Next, we show:
\begin{Lemma} We have
$  I' = \left(  \wedge^2 X  ,  \,\,\,Tr(A) + 2 \pi ,  \,\,\, B_2 J_{n-r} B_1^{t}- A J_{2r}\right) + \mathcal{I}' $, 
where $\mathcal{I}'$ is the ideal generated by the relations $(1)$-$(6)$ from Subsection $4.1$.
\end{Lemma}
\begin{proof}
Using the block form of the matrix $X$ and the relation $ X^{t} S_1 X  +2 (S_0 X + \pi S_1 X   ) = 0 $, it suffices to prove that:
\begin{enumerate}[label=(\alph*)]
    \item $ A^{t}J_{2r}B_1 + 2 \pi J_{2r} B_1 = 0$,
    \item $ A^{t}J_{2r}B_2 + 2 \pi J_{2r} B_2 = 0$,
    \item $ A^{t}J_{2r}A + 2 \pi J_{2r} A = 0$,
\end{enumerate}
in the quotient ring of $S$ by $\left(  \wedge^2 X  ,  \,\,\,Tr(A) + 2 \pi ,  \,\,\, B_2 J_{n-r} B_1^{t}- A J_{2r}\right) + \mathcal{I}' $.

We first discuss (a). Recall that $ A =  B_2 J_{n-r} B_1^{t}J_{2r}, $ $ A B_1 = Tr(A) B_1$ and $ Tr(A) + 2 \pi = 0.$ Thus,
\begin{eqnarray*}    
A^{t}J_{2r}B_1 + 2 \pi J_{2r} B_1  &=& J_{2r} B_1 J_{n-r}B^{t}_2  J_{2r}B_1 + 2 \pi J_{2r} B_1 \\ 
 &=& J_{2r}A B_1 + 2 \pi J_{2r} B_1 \\
 &=& J_{2r} Tr(A) B_1 + 2 \pi J_{2r} B_1 = 0.
\end{eqnarray*}  
Using similar arguments we can prove that the relations (b) and (c) hold.
\end{proof}

The final step is to look more carefully at the minors that come from $ \wedge^2 X$.
\begin{Lemma}\label{minors}
\ \  $ \wedge^2 X \in \mathcal{I}' + \left(   \wedge^2 (B_1 \brokenvert B_2)  , \, \,  Tr(A) + 2 \pi , \, \,  B_2 J_{n-r} B_1^{t}- A J_{2r}\right).$
\end{Lemma}

\begin{proof}
In the proof, we use phrases like: ``minors \textit{only} from $E_c$'', ``minors \textit{only} from $A$ and $B_{\ell}$'', or ``minors from $A$ and $E_c$''. Let us explain what we mean by these terms. Consider the minor
 \[
m^{i, j}_{t,s}= \begin{pmatrix} x_{i, j} & x_{i, s}\\ 
x_{t, j}& x_{t, s}\\
\end{pmatrix}=
x_{i,j}x_{t,s} -x_{i,s}x_{t,j} .
\]
 When we say that ``the minor comes \textit{only} from $E_c$'' we mean that \textit{all} of the entries $\{x_{i,j},x_{t,s},x_{i,s},x_{t,j}\}$ are entries of $E_c$ for $ c \in \{1, 2, 3, 4\} $. Similarly, when we say ``the minor comes \textit{only} from $A$ and $B_{\ell}$'' we mean that \textit{all} of $\{x_{i,j},x_{t,s},x_{i,s},x_{t,j} \}$ are entries either of $ A$ or of $ B_{\ell}$ and \textit{at least one} of the
$\{x_{i,j},x_{t,s},x_{i,s},x_{t,j} \}$ is an entry of $A$ and \textit{at least one} of the $\{x_{i,j},x_{t,s},x_{i,s},x_{t,j} \}$ is an entry of $B_{\ell}$. On the other hand, when we say that ``the minor comes from $A$ and $ E_c$'' we mean that \textit{at least one} of the $\{x_{i,j},x_{t,s},x_{i,s},x_{t,j} \}$ is an entry of $ A$ and \textit{at least one} of the $\{x_{i,j},x_{t,s},x_{i,s},x_{t,j} \}$ is an entry of $ E_c$ for $ c \in \{1, 2, 3, 4\} $.
We have the following cases of minors: 
\begin{enumerate}
            \begin{multicols}{2}
            \item \textit{only} from $E_c$ 
            
            \item \textit{only} from $A$
            
            \item \textit{only} from $O_m$
            
            \item from $E_c$ and $A$ 
            
            \columnbreak
            \item \textit{only} from $A$ and $B_{\ell}$
            
            \item \textit{only} from $E_c$ and $O_m$
            
            \item \textit{only} from $A$ and $O_m$
            
            \item \textit{only} from $E_c$ and $B_{\ell}$
            
            \end{multicols}
        \end{enumerate}
In each case, we will show that the corresponding minors belong to
$$
 \mathcal{I}' + \left(   \wedge^2 (B_1 \brokenvert B_2)  , \, \,  Tr(A) + 2 \pi , \, \,  B_2 J_{n-r} B_1^{t}- A J_{2r}\right). $$
We will start by considering case ($1$), i.e. minors \textit{only} from $ E_c$. It suffices to prove $ x_{i,j}x_{t,s} =x_{i,s}x_{t,j}$ in the quotient ring 
$$ \frac{S}{\mathcal{I}' + \left(   \wedge^2 (B_1 \brokenvert B_2)  , \, \,  Tr(A) + 2 \pi , \, \,  B_2 J_{n-r} B_1^{t}- A J_{2r}\right)} .$$
By using minors from $ B_{\ell}$ for $ \ell \in \{1, 2\} $ and for all $i,j \in Z^{c}$, we have the following equation in the above quotient ring: 
    $$ \bigg( \sum_{a= n-(r-1)}^{d-(n-r)} \frac{1}{2}x_{d+1-a,d+1-i}x_{a,j} \bigg) \bigg( \sum_{a= n-(r-1)}^{d-(n-r)} \frac{1}{2} x_{d+1-a,d+1-t}x_{a,s}\bigg)= $$ $$ \bigg( \sum_{a= n-(r-1)}^{d-(n-r)} \frac{1}{2} x_{d+1-a,d+1-i}x_{a,s} \bigg) \bigg( \sum_{a= n-(r-1)}^{d-(n-r)} \frac{1}{2} x_{d+1-a,d+1-t}x_{a,j}\bigg).$$
By using the relations ($1$)-($4$) from Subsection $4.1$ we can express the entries $ x_{i,j}$ of $E_c$ as: 
$$ x_{i,j} = -\sum_{a= n-(r-1)}^{d-(n-r)} \frac{1}{2} x_{d+1-a,d+1-i}x_{a,j}  $$
with $ i,j \in Z^{c}$. Using this and the above equality we obtain:
$$  x_{i,j}x_{t,s} =x_{i,s}x_{t,j}.$$
The rest of cases ($2$)-($8$) can be handled by similar arguments. More precisely, by using the relations ($1$)-($6$) from Subsection $4.1$ we can express all the entries from $E_i$ for $ i \in \{1,2,3,4\}$ and $O_1,O_2$ in terms of the entries of $ B_1,B_2$. Also, by using $A = B_2 J_{n-r} B_1^{t}J_{2r}$ we can express all the entries of $A$ in terms of the entries of  $ B_1,B_2$. After that, by using the $2\times2$-minors from the matrix $(B_1|B_2)$ we get the desired result in all the remaining cases.  
\end{proof}

{\sl End of proof of Theorem \ref{simpler}:} From the above lemma we obtain that 
$$ I' = \mathcal{I}' + \left(   \wedge^2 (B_1 \brokenvert B_2)  , \, \,  Tr(A) + 2 \pi , \, \,  B_2 J_{n-r} B_1^{t}- A J_{2r}\right).$$
Observe that an equivalent way of writing $I'$ is:
$$ I' = \mathcal{I}' + \left(   \wedge^2 (B_1 \brokenvert B_2)  , \, \,  Tr(B_2 J_{n-r} B_1^{t}J_{2r}) + 2 \pi , \, \,  B_2 J_{n-r} B_1^{t}- A J_{2r}\right).$$
Using this and the fact that $I'=I $ the proof of Theorem \ref{simpler} follows.\end{proof}

\section{Flatness}\label{flatness}
We continue to assume $d=2n$, $l=2r$.

Recall that $U_{d,l} = \Spec(S/I)$. The goal of this section is to prove that $U_{d,l}$ is flat over $\mathcal{O}$ as given by Theorem \ref{Flat}. The simplification that we obtained from Theorem \ref{simpler} quickly gives the following, which in turn plays a crucial role for the proof of Theorem \ref{Flat}. Then the proof of Theorem \ref{mthm1} easily follows.

\begin{Theorem}\label{C-M mixed}
$U_{d,l}$ is Cohen-Macaulay.
\end{Theorem}

\begin{proof}
Denote by $ \mathcal{O}[B_1|B_2]$ the polynomial ring over $\mathcal{O}$ with variables the entries of the matrix $(B_1|B_2)$. From Theorem \ref{simpler} we obtain the isomorphism 
$$
 \frac{S}{I} \cong \frac{ \mathcal{O}[B_1|B_2]}{\left(\wedge^2 (B_1 \brokenvert B_2) , Tr(B_2 J_{n-r} B_1^{t}J_{2r})+2\pi \right)}.
 $$
 Set $\displaystyle \mathcal{R}:=  \mathcal{O}[B_1|B_2] / \big(\wedge^2 (B_1 \brokenvert B_2) \big) $. By  \cite[Remark  2.12]{DR}, the ring $\mathcal{R}$ is Cohen-Macaulay and an integral domain. We consider the point $P$ of the determinantal variety which is defined by the relations:
 $$
\begin{pmatrix} x_{n-r+1, 1} & x_{n-r+1, d}\\ 
x_{n+r, 1}& x_{n+r, d}\\
\end{pmatrix}= \begin{pmatrix} 1-\pi & 1-\pi\\
1 & 1 \\ 

\end{pmatrix} 
 $$
and we set all the other variables equal to zero. We can easily observe that $Tr(B_2 J_{n-r} B_1^{t}J_{2r})+2\pi$ is not zero over the point $P$. Therefore, we have 
$$ht((Tr(B_2 J_{n-r} B_1^{t}J_{2r})+2\pi)) = 1.$$
We apply the fact that  if $A$ is Cohen-Macaulay and the ideal $I= (a_1,\ldots,a_r) $ of $A$ has height $r$ then $A/I$ is Cohen-Macaulay (\cite[example 3 (16.F)]{Ma})  to $A=\mathcal{R}$ and $a_1=Tr(B_2 J_{n-r} B_1^{t}J_{2r})+2\pi$. We obtain that
$$ \frac{  \mathcal{O}[B_1|B_2] }{\left(\wedge^2 (B_1 \brokenvert B_2) , Tr(B_2 J_{n-r} B_1^{t}J_{2r})+2\pi\right)}$$
is Cohen-Macaulay.  This implies the result.
\end{proof}
\begin{Remark}\label{dimension}
From the above proof and the standard formula for the dimension of the determinantal varieties (see \cite[Proposition 1.1]{DR}) we obtain that $ \text{dim}(S/I) = \text{dim} (\mathcal{R}) -1 = d-1. $ Hence, the dimension of $ U_{d,l}$ is $d-1.$
\end{Remark}

\begin{Remark}\label{C-M remark}
Mimicking the proof of Theorem \ref{C-M mixed} and by considering Remark \ref{dimension} we obtain that the special fiber $ \overline U_{d,l} $ of $U_{d,l}$ is Cohen-Macaulay and of dimension $d-2.$ 
\end{Remark}

\begin{Theorem}\label{Flat}
$U_{d,l}$ is flat over $\mathcal{O}$ and of relative dimension $d-2$.
\end{Theorem}

\begin{proof}
 From Remark \ref{dimension} we have that $U_{d,l}$ has dimension $d-1.$ From Remark \ref{C-M remark} we have that the dimension of $\overline U_{d,l}$ is $d-2.$ Using the fact that $U_{d,l}$ is Cohen-Macaulay, see Theorem \ref{C-M mixed}, we obtain that $ht((\pi))=1$. Hence, we have that $(\pi)$ is a regular sequence i.e. $\pi$ is not a zero divisor (see \cite{Ho}). From this, flatness of $U_{d,l}$ follows.
 \end{proof}
$\quad$\\
{\sl Proof of Theorem \ref{mthm1}.} 
Given the arguments in Sections 4 and 5 above, it just remains to show that $U_{d,l} \subset \Mloc(\Lambda) $.
    
    From Lemma \ref{GenericFibers} we see that the generic fibers of $ \mathcal{U} = \Mloc(\Lambda)\cap \mathcal{U}^{\text{naive}}$ and $ \mathcal{U}^{\text{naive}}$ are not equal. Recall that $ \mathcal{U}$ is a $\mathcal{O}$-flat scheme of relative dimension $d-2$. Let $M_0$ be the open subscheme of $\mathcal{U}^{\text{naive}}$ which is the complement of
the locus where $\mathcal{L} = \mathbb{L}_{R}$. The generic fiber of $M_0$ agrees with $ \mathcal{U}\otimes_{\mathcal{O}} \breve F$. From Theorem 5.0.4, we have that $U_{d,l}$ is the $\mathcal{O}$-flat closed subscheme of $\mathcal{U}^{\text{naive}} $ of relative dimension $d-2$. We can easily see that the $\breve F$-point $\mathcal{L} = \mathbb{L}_{\breve F}$ does not belong in the generic fiber of $ U_{d,l}$.
        
Combining all the above, we see that the $\mathcal{O}$-flat schemes $ \mathcal{U} $ and $U_{d,l} $ are of relative dimension $d-2$ and have the same generic fiber. Thus, $U_{d,l}\otimes_{\mathcal{O}} \breve F \subset \Mloc(\Lambda)\otimes_{\mathcal{O}} \breve F $ and $U_{d,l} \subset \Mloc(\Lambda) $.\qed
 
\section{Special Fiber }\label{special fiber}
  
In all of Section \ref{special fiber}, we assume $d = 2n$ and $l = 2r$.

We will prove that the special fiber $\overline U_{d,l} $ of $U_{d,l} = \Spec(S/I)$ is reduced; see Section \ref{sect4} for undefined terms.\\
$\quad$\\
{\sl Proof of Theorem \ref{spec.fib.red.}.}
From Remark \ref{C-M remark} we know that $\overline U_{d,l} $ is Cohen-Macaulay. By using Serre's criterion for reducedness, it suffices to prove that the localizations at minimal primes are reduced. From Lemma \ref{prime r=1} and Lemma \ref{prime r>1} we obtain the minimal primes of $\overline U_{d,l} $. Below, we focus on the localization of $\overline U_{d,l} $ over $I_1$ for $1<r<n-1$ (see Subsection \ref{r>1} for the notation). In the other cases the proof is similar.

We first introduce some additional  notation:

Denote by $ k[B_1|B_2]$ the polynomial ring over $k$ with variables the entries of the matrix $(B_1|B_2)$.
We set  $ Z_1 :=  Z \setminus \{ n-r+1\} $ and $Z_2 := Z^{c}\setminus\{1\}$.  By direct calculations we get
\[
 Tr(B_2 J_{n-r} B_1^{t}J_{l})= \sum_{n-r+1\leq i\leq n+r, 1\leq j\leq n-r} x_{i, j}\, x_{d+1-i, d+1- j}.
 \]
Set $t_r :=\frac{1}{2} Tr(B_2 J_{n-r} B_1^{t}J_{l}) $ and $t'_r = t_r -  x_{n-r+1, 1}x_{n+r, d}$. Lastly, we set $m:=  x_{n-r+1,d}^{-1} t'_r $.
We refer the reader to Subsection \ref{r>1} for the rest undefined terms.

From Theorem \ref{simpler} and Lemma \ref{sym} we obtain that the special fiber $\overline U_{d,l} $ is given by the quotient of $k[B_1|B_2]$ by the ideal $(\wedge^2 (B_1 \brokenvert B_2) , t_r )$.
Set $J_1 = I_1 + \left(\wedge^2 (B_1 \brokenvert B_2) ,\,\,  t_r \right) $ where 
$$ I_1 = \bigg(( \sum_{a= n-(r-1)}^{n} x_{d+1-a,d+1-i}x_{a,j})_{i,j \in Z^{c}}, \,\,  \wedge^2 (B_1 \brokenvert B_2) \bigg).$$

By localizing $ k[B_1|B_2]/ (\wedge^2 (B_1 \brokenvert B_2) , t_r) $ at $ J_1$ we have 
 $$ \left(\frac{ k[B_1|B_2]}{\wedge^2 (B_1 \brokenvert B_2) ,  t_r}\right)_{ J_1} \cong  \frac{ k[B_1|B_2]_{I_1}}{\left(\wedge^2 (B_1 \brokenvert B_2) , t_r \right)_{I_1}}  .$$
 In the proof of Lemma \ref{prime r>1} we used the fact that $ x_{n-r+1,1} \notin I_1.$ Similarly, we have that $  x_{n-r+1,d} \notin I_1.$ 
\smallskip

\noindent {\bf Claim:} {\sl  $$ \left(\wedge^2 (B_1 \brokenvert B_2) ,\,\, t_r\right)_{I_1}$$ $$= \bigg( (x_{i,j} - x_{n-r+1,1}^{-1}x_{i,1} x_{n-r+1,j})_{\substack{i \in  Z_1, \, j \in Z_2} },\,\, (x_{n+r,1}+m)  \bigg)_{I_1}  .$$} 
\smallskip

 {\sl Proof of the claim:}
 From the minors we have $x_{i,j} x_{n-r+1,1} - x_{i,1} x_{n-r+1,j} = x_{n-r+1,1}(x_{i,j} - x_{n-r+1,1}^{-1}x_{i,1} x_{n-r+1,j}).$
We rewrite $ t_r$ as: 
$$ t_r= x_{n-r+1,1}x_{n+r,d}+ t'_r. $$
Combining the above relation with the minor $ x_{n-r+1,1}x_{n+r,d} =  x_{n-r+1,d}x_{n+r,1}$ 
we obtain 
$$ x_{n-r+1,1}x_{n+r,d}+ t'_r = x_{n-r+1,d}(x_{n+r,1}+ m).$$
Now, because $x_{n-r+1,1}, x_{n-r+1,d} \notin I_1 $ the claim follows. Combining all the above we have:
 \begin{eqnarray*}  
 && \frac{ k[B_1|B_2]_{I_1}}{\left(\wedge^2 (B_1 \brokenvert B_2) ,\,\, t_r\right)_{I_1}}  \\  
 &\cong &\frac{k[B_1|B_2]_{I_1}}{\bigg( (x_{i,j} - x_{n-r+1,1}^{-1}x_{i,1} x_{n-r+1,j})_{\substack{i \in  Z_1,\,  j \in Z_2} },\,\, (x_{n+r,1}+m)  \bigg)_{I_1}  } \\
 &\cong&  k [(x_{i,1})_{ i \in Z \setminus \{ n+r \}},\,\, (x_{n-r+1,j})_{j \in Z^{c}},\,\,x_{n-r+1,1}^{-1},\,\,x_{n-r+1,d}^{-1}]_{J_1}  
\end{eqnarray*} 
 where the last one is a reduced ring. Thus, the special fiber of $U_{d,l}$ is reduced.\qed

\section{Irreducible Components}\label{comp.}

We continue to assume $d=2n$, $l=2r$.

Recall that $U_{d,l} = \Spec(S/I)$ and $\overline U_{d,l} $ is the special fiber of $U_{d,l} $. In this section, the main goal is to calculate the irreducible components of $\overline U_{d,l} $.
\begin{Theorem}\label{irred}

(i) For $r=1$ and $r=n-1$, $\overline U_{d,l} $ has three irreducible components.
  
   (ii)  For $1<r<n-1$, $\overline U_{d,l} $ has two irreducible components.
   
\end{Theorem}

The proof of the theorem will be carried out in Subsection \ref{r=1} (case ($i$)) and in Subsection \ref{r>1} (case ($ii$)).
\subsection{Case (i)}\label{r=1}
In this subsection we will prove Theorem \ref{irred} in the case $r=1$. A similar argument works in the case $r=n-1$. For this subsection we introduce the following notation. Observe that when $r=1$, $Z  = \{ n,n+1\}  $ and $ Z^{c} = \{ 1,2,3, \ldots, d\} \setminus Z $. We  rename the variables as follows:  
$$
v_i = x_{n,i}, \ \hbox{\rm for} \, \,  i\in Z^{c} \ \hbox{\rm and}\, w_j = x_{n+1,j}, \ \hbox{\rm for}\,\,  j \in Z^{c} .
$$
Define the polynomial ring $ \displaystyle S_{\text{sim}}= k[v_i,w_j]_{i,j\in Z^{c}} $. From Theorem \ref{simpler} we obtain that the special fiber is isomorphic to $S_{\text{sim}} /  I_{\text{sim}}$ where 
$$
I_{\text{sim}} = \Bigg( \sum_{i=1}^{n-1} v_{i}w_{2n+1-i}, \quad (v_{i}w_{j}-v_{j}w_{i})_{i,j\in Z^{c}}\Bigg).
$$
It is not very hard to observe, by using the minors, that 
$$ 
I_{\text{sim}} = \Bigg( \sum_{i=1}^{n-1} v_{i}w_{2n+1-i}, \quad(v_{i}\sum_{j=1}^{n-1} w_{j}w_{2n+1-j})_{i\in Z^{c}},\quad(w_{i} \sum_{j=1}^{n-1} v_{j}v_{2n+1-j})_{i\in Z^{c}},
$$
$$
\sum_{i=1}^{n-1} v_{i}v_{2n+1-i} \sum_{j=1}^{n-1} w_{j}w_{2n+1-j},\quad (v_{i}w_{j}-v_{j}w_{i})_{i,j\in Z^{c}}\Bigg). 
$$ 
Next, we set
$$
I_1= \Bigg( (v_{i})_{i\in Z^{c}} \Bigg), \quad  I_2 = \Bigg( (w_{i})_{i\in Z^{c}}\Bigg)
$$
and
$$
I_3 = \Bigg(\sum_{i=1}^{n-1} v_{i}v_{2n+1-i}, \quad \sum_{i=1}^{n-1} w_{i}w_{2n+1-i} , \quad
\sum_{i=1}^{n-1} v_{i}w_{2n+1-i}, \quad  (v_{i}w_{j}-v_{j}w_{i})_{i,j\in Z^{c}}   \Bigg) .
$$
{\sl Proof of Theorem \ref{irred} $(i)$:} From the above, it suffices to calculate the irreducible components of $V(I_{\text{sim}}).$
 
Observe that the elements
$$
(v_{i}\sum_{j=1}^{n-1} w_{j}w_{2n+1-j})_{i\in Z^{c}},\quad(w_{i} \sum_{j=1}^{n-1} v_{j}v_{2n+1-j})_{i\in Z^{c}}, \quad \sum_{i=1}^{n-1} v_{i}v_{2n+1-i} \sum_{j=1}^{n-1} w_{j}w_{2n+1-j}
$$
belong to $I_{\text{sim}}$. Therefore, we can easily see that 
$$ V( I_{\text{sim}}) = V(I_1) \cup V(I_2) \cup V(I_3).$$
Observe also that 
$$
S_{\text{sim}}/I_1 \cong k [(w_{j})_{j \in Z^{c} }], \qquad  S_{\text{sim}}/I_2 \cong k [(v_{j})_{j \in Z^{c} }]. $$
Thus, the closed subschemes $ V(I_1),V(I_2)$ are affine spaces of dimension $d-2$ and so they are irreducible and smooth. We have to check that the third one is irreducible and of dimension $d-2$. Notice that $I_3 $ is generated by homogeneous elements. Thus, it suffices to prove that the projectivization 
$$ V_p(I_3) \subseteq \mathbb{P}_k^{2d-5} $$
of the affine cone $ V(I_3)$ is irreducible. Consider
$$ V_{v_1} : = V_p(I_3) \cap U_{v_1}, $$
where $ U_{v_1} = D( v_{1} \neq 0) $. We can see that it is isomorphic to 
$$
\Spec( k [ (\frac{v_{i}}{v_1})_{ i \in Z^{c} \setminus \{ 1,d \} },\frac{w_{1}}{v_1}]),
$$
and so it is irreducible. By symmetry   we  have a similar result for every $ V_{v_i}$ and $ V_{w_j}$ with $ i,j \in Z^{c}$. Moreover, the $ V_{v_i}$ and $ V_{w_j}$ form a finite open cover of irreducible open subsets of $V_p(I_3)$. Thus $V_p(I_3)$ is irreducible and so $ V(I_3) \subseteq \mathbb{A}_k^{2d-4}$ is irreducible. This completes the proof
of Theorem \ref{irred} $(i)$. \qed

We can go one step further and prove that:
\begin{Lemma}\label{prime r=1}
 The ideals $I_1 , I_2 , I_3$ are prime.
\end{Lemma}
\begin{proof}
 From the proof of Theorem \ref{irred} $(i)$, $I_1,I_2$ are clearly prime ideals. It suffices to prove that $I_3$ is prime.  From Theorem \ref{irred} $(i)$ we have that $I_3$ is a primary ideal and so every zero divisor in $ D :=    k[v_i,w_j]_{i,j\in Z^{c}} /I_3$ is a nilpotent element. Hence, it suffices to prove that $ D$ is reduced. Serre's criterion for reducedness (\cite{Ma} $17.$I) states: the ring $D$ is reduced if and only if $D$ satisfies ($S_1$) and ($R_0$). From the proof of Theorem 7.0.1 $(i)$ we have that the scheme $  \Spec(D)$ is smooth over $\Spec(k)$ outside from its closed subscheme of dimension $0$ which is defined by the ideal $( (v_{j})_{j \in Z^{c}},(w_{i})_{i \in Z^{c} } ). $ Thus, the condition ($R_0$) is satisfied. So, it suffices to prove: $\text{min}(1,\text{ht}(\mathfrak{p}))\leq \text{depth}D_{\mathfrak{p}}$. The interesting primes $ \mathfrak{p}$ are those which contain $( (v_{j})_{j \in Z^{c}},(w_{i})_{i \in Z^{c} } )$, because otherwise we are in the smooth case so the desired inequality is satisfied. Thus, it suffices to find an element $f $ in the ideal $ ( (v_{j})_{j \in Z^{c}},(w_{i})_{i \in Z^{c} } )$ which is regular in the quotient ring $D$.
\smallskip

\noindent {\bf Claim:} {\sl The element $w_1$  is regular in $D$ and so we can take $f=w_1$.}
\smallskip

 {\sl Proof of the claim:} Assume   that $w_{1}$ is not a regular element in $D$. Then, because $I_3$ is primary, we have $ w_{1}^m \in I_3$ for some $m>0$  (see \cite{Ma} $1.$A). We will obtain a contradiction by using  Buchberger's algorithm. This is a method of transforming a given set of generators for a polynomial ideal into a Gröbner basis with respect to some monomial order. For more information about this algorithm we refer the reader to   \cite[Chapter 2]{Cox}. 
 
Set $R= k[v_i,w_j]_{i,j\in Z^{c}}$ for the polynomial ring.
We choose the following order for our variables:
$$ v_{1}>v_{2}>\cdots >v_{d}>w_{1}>\cdots >w_{d}.$$ 
Then, the graded lexicographic order  induces an order of all  monomials in $R$.

Next, we recall the division algorithm in $ R$. We fix the above monomial ordering. Let $J = ( f_1, \ldots , f_s)$ be an ordered $s$-tuple of polynomials in $R$. Then
every $g \in R$ can be written as
 $$g= a_1 f_1 + \ldots + a_s f_s + r,$$
where $a_i,r \in R$ and either $r = 0$ or $r$ is a linear combination, with coefficients in $k$, of monomials, none of which is divisible by any of $LT( f_1), \ldots , LT( f_s)$. By $LT(f_i)$ we denote the leading term of $f_i$. We will call $r$ a remainder of $g$ on division by $J$. (See     \cite[Chapter 2]{Cox} for more details.)

Recall that the $\mathcal{S}$-polynomial of the pair $f$, $g\in R$ is
$$
 \mathcal{S}(f,g) = \frac{LCM(LM(f),LM(g))}{LT(f)}f - \frac{LCM(LM(f),LM(g))}{LT(g)}g.
 $$ 
 Here, by $LM(f)$ we denote the leading monomial of $f$ according to the above ordering. 
 
To find the Gröbner basis for the ideal $ I_3$, we start with the generating set $$\Bigg\{ \sum_{i=1}^{n-1} v_{i}v_{2n+1-i}, \sum_{i=1}^{n-1} w_{i}w_{2n+1-i} , \sum_{i=1}^{n-1} v_{i}w_{2n+1-i},(v_{i}w_{j}-v_{j}w_{i})_{i,j\in Z^{c}} \Bigg\}  .$$
Then, we calculate all the $\mathcal{S}$-polynomials $\mathcal{S}(f,g)$, where $f,g$ are any two generators from the generating set that we have started. If all the $\mathcal{S}$-polynomials are divisible by the generating set then  the generating set already forms a Gröbner basis. On the other hand, if a remainder is nonzero we extend our generating set by adding this remainder and we repeat the above process until we have a generating set where all the $\mathcal{S}$-polynomials are divisible by the generating set. In our case, the generators of $I_3$ are homogeneous  polynomials of degree $2$. The monomials of those homogeneous polynomials have one of the following forms: \begin{enumerate}
    \item $ v_i v_j$ with $ i \neq j$, or
    \item $ w_i w_j$ with $ i \neq j$, or
    \item $ w_i v_j$ with $ i \neq j$. 
\end{enumerate}
Thus, the nonzero remainder of any $\mathcal{S}$-polynomial is a polynomial where each monomial is divisible by at least one monomial of the above three forms. By this observation we can see that, the Gröbner basis cannot have a monomial that looks like $ w^m_i$ or $ v^m_j$. Now, let $\{g_1 \ldots, g_{N}\}$ be the Gröbner basis of $I_3$. From the above we have $$ w^{m}_1 \notin \langle LT(g_1),\ldots ,LT(g_{N}) \rangle .$$ 
Moreover, because $\{g_1 \ldots, g_{N}\}$ is the Gröbner basis of $I_3$ we have that
$$  \langle LT(g_1),\ldots ,LT(g_{N}) \rangle  = LT(I_3),$$
(see   \cite[Chapter 5]{Cox}). By $LT(I_3) $ we denote the ideal generated by all the leading terms of elements of $I_3$. Therefore,
$$ w^{m}_1 \notin LT(I_3).$$
Hence, $ w_{1}$ is a regular element in $D$ and so $I_3$ is a prime ideal. This completes the proof of the claim
and by the above the proof of lemma.
\end{proof}
\subsection{Case (ii)}\label{r>1}
In this subsection we will prove Theorem \ref{irred} in the case $1<r<n-1.$ In this case we have $Z = \{ n-(r-1), \ldots,n,n+1,  \ldots ,d-n+r\}$ and $ Z^{c} = \{ 1,2,3, \ldots, d\} \setminus Z $. For the undefined terms below we refer the reader to Section \ref{sect4} and \ref{special fiber}. From Theorem \ref{simpler} we obtain that the special fiber $\overline U_{d,l} $ is given by the quotient of $k[B_1|B_2]$ by the ideal $ I_s =  (\wedge^2 (B_1 \brokenvert B_2) , t_r )$. Also with direct calculations and by using the minors we can see that 
$$
I_s = \bigg(\wedge^2 (B_1 \brokenvert B_2) ,\, t_r , \, \bigg( \sum_{a= n-(r-1)}^{n} x_{d+1-a,d+1-t}x_{a,s} \sum_{b=1}^{n-r} x_{i,b}x_{d+1-j,d+1-b}  \bigg)_{\substack{i,j \in Z \\ t,s \in Z^{c}}}  \bigg). 
$$
Next, set 
$$ I_1= \Bigg(  \bigg( \sum_{a= n-(r-1)}^{n} x_{d+1-a,d+1-t}x_{a,s}\bigg)_{t,s \in Z^{c}}, \quad \wedge^2 (B_1 \brokenvert B_2) \bigg) $$
and
$$ I_2 = \Bigg( \bigg(\sum_{a=1}^{n-r} x_{i,a}x_{d+1-j,d+1-a}  \bigg)_{i,j \in Z}, \quad \wedge^2 (B_1 \brokenvert B_2) \Bigg).$$
 {\sl Proof of Theorem \ref{irred} $(ii)$:} From the above, it suffices to calculate the irreducible components of $V(I_s)$.

Observe that
$$ \bigg( \sum_{a= n-(r-1)}^{n} x_{d+1-a,d+1-t}x_{a,s} \sum_{b=1}^{n-r} x_{i,b}x_{d+1-j,d+1-b}  \bigg)_{\substack{i,j \in Z \\ t,s \in Z^{c}}}  \in I_s.$$
Therefore, we can easily see that 
$$ V( I_s) = V(I_1) \cup V(I_2).$$
Next, we prove that the closed subschemes $ V(I_1)$ and $V(I_2)$ are irreducible of dimension $ d-2.$
We will start by proving that $V(I_1)$ is an irreducible component. Observe that $I_1 $ is generated by homogeneous elements. Thus, it suffices to prove that the projectivization 
$$ 
V_p(I_1) \subseteq \mathbb{P}_k^{4r(n-r)-1} .
$$
of the affine cone $ V(I_1)$ is irreducible. We look at 
 $$ V_{x_{n-(r-1),1}} : = V_p(I_1) \cap U_{x_{n-(r-1),1}}, $$
 where $ U_{x_{n-(r-1),1}} = D( x_{n-(r-1),1} \neq 0) $. We can see that it is isomorphic to 
 $$ \Spec ( k [ (\frac{x_{n-(r-1),j}}{x_{n-(r-1),1}})_{j\in Z^{c}\setminus\{1\}},(\frac{x_{t,1}}{x_{n-(r-1),1}})_{ n-(r-2)\leq t \leq n+r-1}])$$
 and so it is irreducible. Because of the symmetry of the relations we will have a similar result for every $ V_{x_{t,s}}$ with $ t \in Z, s \in Z^{c} $. Moreover,
 $$ (V_{x_{t,s}})_{\substack{t \in Z, \,  s \in Z^{c}}}$$ 
 form a finite open cover of irreducible open subsets of $V_p(I_1)$ and thus we get that $V_p(I_1)$ is irreducible and so $ V(I_1) $ is irreducible of dimension $d-2$.
 
 We use a similar argument for $V(I_2)$. $I_2 $ is generated by homogeneous elements and so by using the projectivization 
  $$ V_p(I_2) \subseteq \mathbb{P}_k^{4r(n-r)-1} $$
  of the affine cone $ V(I_2)$, it suffices to prove that $V_p(I_2)$ is irreducible. We look at 
  $$ V_{x_{n-(r-1),1}} : = V_p(I_2) \cap U_{x_{n-(r-1),1}}, $$
  where $ U_{x_{n-(r-1),1}} = D( x_{n-(r-1),1} \neq 0) $. We can see that it is isomorphic to 
  $$ \Spec (  k [ (\frac{x_{n-(r-1),j}}{x_{n-(r-1),1}})_{j\in Z^{c}\setminus\{1,d\}},(\frac{x_{t,1}}{x_{n-(r-1),1}})_{ n-(r-2)\leq t \leq n+r} ]) $$
  and so it is irreducible. Therefore, $V(I_2)$ is irreducible of dimension $d-2$. This completes the proof of Theorem \ref{irred} $(ii)$. \qed
 
Next, we prove that:
\begin{Lemma}\label{prime r>1}
The ideals $I_1, I_2$ are prime.
\end{Lemma} 
\begin{proof}
To see that $I_1, I_2$ are prime ideals one proceeds exactly as in Lemma \ref{prime r=1}. So, it suffices to find a regular element $f$ such that $f \in ( (x_{t,s})_{ t \in Z, s \in Z^{c} } )$.  We claim that $x_{n-r+1,1}$ is a choice for $f$. Assume for contradiction that $x_{n-r+1,1}$ is not a regular element. Then, because $I_i$ is primary, we should have that $ x_{n-r+1,1}^m \in I_i$ for some $m>0.$

We choose the following order for our variables $ (x_{t,s})_{ t \in Z, s \in Z^{c} }$:
$$ x_{n-r+1,1}>\ldots>x_{n-r+1,n-r}>x_{n-r+1,n+r+1}>\ldots>  x_{n-r+1,d}> x_{n-r+2,1}>$$ $$\ldots>x_{n-r+2,d}> \ldots >x_{n+r,1}>\ldots>x_{n+r,n-r}>x_{n+r,n+r+1}>\ldots>x_{n+r,d}. $$ 
Then, the graded lexicographic order induces an ordering to all the monomials. First, let's consider the ideal $I_1 $. In order to find the Gröbner basis for $ I_1$, we start with the generating set 
$$  \Bigg\{ \bigg( \sum_{a= n-(r-1)}^{n} x_{d+1-a,d+1-i}x_{a,j}\bigg)_{i,j \in Z^{c}}, \quad  \wedge^2 (B_1 \brokenvert B_2)  \Bigg\}.$$
After that we calculate all the $\mathcal{S}$-polynomials $\mathcal{S}(f,g)$, where $f,g$ are any two generators from the generating set that we have started; so in our case is $I_1$. The generators of $I_1$ are homogeneous  polynomials of degree $2$. The monomials of those homogeneous polynomials have one of the following form:
$$ (x_{i,j}x_{t,s} )$$
with $ i,t \in Z$, $ j,s \in Z^{c}$ and either $ i \neq t$ or $ j \neq s$. Thus, the nonzero remainder of any $\mathcal{S}$-polynomial is a polynomial where each monomial is divisible by at least one monomial of the above form. By this observation we can see that, the Gröbner basis cannot have a monomial that looks like $x_{n-r+1,1}^m$. Now, by using a Gröbner basis argument as in Lemma \ref{prime r=1} we deduce that $ x_{n-r+1,1}$ is a regular element and so $I_1$ is a prime ideal.

Now, by looking the ideal $I_2$ we have the generating set:
$$ \Bigg\{  \bigg(\sum_{a=1}^{n-2} x_{i,a}x_{d+1-j,d+1-a}   \bigg)_{i,j \in Z}, \quad \wedge^2 (B_1 \brokenvert B_2)\Bigg\}.$$
So, we observe that in this case also the generators of the ideal $ I_2$ are homogeneous polynomials of degree $2$. All the monomials have one of the following form: 
$$ (x_{i,j}x_{t,s} )$$
with $ i,t \in Z$, $ j,s \in Z^{c}$ and either $ i \neq t$ or $ j \neq s$. So, by using the same argument we can prove that $I_2$ is a prime ideal.
\end{proof}

\medskip

\section{The Remaining Cases}\label{remaining} 
In this section we sketch the proof of Theorems \ref{mthm1} and \ref{mthm2} in the remaining cases of parity for $d$ and $l$. The main point is that in all the cases the affine chart $U_{d,l} $ of the local model is displayed as a hypersurface in a determinantal scheme. The arguments are similar with the proof of the case $(d,l)= (\even,\even)$. In fact, in the case that 
$(d,l)= (\odd,\odd)$ the argument is exactly the same. The case of Theorem \ref{mthm2} (different parity) 
is somewhat different as we explain below.

\subsection{Proof of Theorem \ref{mthm2}}\label{prf3}
\begin{proof}
We use the notation from Subsection $3.3$. We also introduce some new notation.
Set 
  \[
  Z' : =  \{ n-r'+1,\ldots,n,n+2, \ldots ,d-n+r'\},\quad  (Z')^{c}: = \{ 1,2,3, \ldots, d\} \setminus Z'.
  \]
   Also, define the polynomial ring 
   \[
    \mathcal{O}[B_1 | Q | B_2] :=  \mathcal{O}[(x_{t,s})_{ t \in Z', s \in (Z')^{c} }].
    \]
     Lastly, set
$$
 \wedge^2 (B_1 \brokenvert Q \brokenvert B_2) := \Big( x_{i,j}x_{t,s} - x_{i,s}x_{t,j} \Big) _{{i,t \in Z' \ j,s \in (Z')^{c}}}. 
 $$  
We can now sketch the proof. Similar elementary arguments as in the proof of Theorem \ref{simpler} give
  that the quotient $\mathcal{O}[X]/I$ is isomorphic to the quotient  of $ \mathcal{O}[B_1 | Q | B_2]$ by the ideal 
 \[
  ( \wedge^2 (B_1 \brokenvert Q \brokenvert B_2), \  Tr( H  B_2 J_{n-r'} B_1^{t} H J'_{l+1} +\frac{1}{2} H  Q Q^{t} HJ'_{l+1} ) + 2 \pi ).
  \]
  This is the corresponding isomorphism to $ S/I\cong S''/I''$ of Theorem \ref{simpler}. So, the above ideal is now playing the role of $I''$.
  
In the course of proving Theorem \ref{simpler} we proved  Lemmas \ref{eq1} and \ref{minors}. The main ingredients for the proofs of these Lemmas are the relations given by the entries of $  B_2 J_{n-r} B_1^{t}- A J_{l}$ and by the element  $ Tr(A) + 2\pi$. The proofs of the corresponding Lemmas \ref{eq1} and \ref{minors} in the case of $d$ and $l$ with opposite parity, are similar to that of $(d,l)=(\text{even},\text{even} )$. In this case the role of $A$ is now played by $A'$ and  the relations $  B_2 J_{n-r} B_1^{t}- A J_{l}$ and $ Tr(A) + 2\pi$ are replaced by $  2 H  B_2 J_{n-r'} B^{t}_1 H  + H  Q Q^{t} H  - 2 H  A J'_{l+1} H $ and $Tr( A') + 2 \pi$.

The rest of the argument deducing flatness of $U_{d,l}$ is the same as before. Lastly, the same argument as in the proof of Theorem \ref{mthm1} shows that  $U_{d,l}\otimes_{\mathcal{O}} \breve F \subset \Mloc(\Lambda)\otimes_{\mathcal{O}} \breve F $ and $U_{d,l} \subset \Mloc(\Lambda) $. See also the proof of \cite[Theorem 5.2.1]{PaZa} for a variation of this proof.
\end{proof}

\section{Appendix}\label{Appendix}
In this section we present the irreducible components of the special fiber $\overline U_{d,l}$ of  $ U_{d,l} $ in the remaining cases. We omit the proofs which are similar to Theorem \ref{irred}. For the notation we refer the reader to Section \ref{sect4} and Section \ref{remaining} .

\subsection{} $(d,l)=(\odd,\odd)$.

\subsubsection{}
 When $l<d-2$, the irreducible components of $\overline U_{d,l}$ are the closed subschemes $ V(I_1)$ and $V(I_2)$ where:
$$
 I_1= \Bigg( 
\bigg( \sum_{a= n-(r-1)}^{n} x_{d+1-a,d+1-t}x_{a,s}+\frac{1}{2}x_{n+1,d+1-t}x_{n+1,s}\bigg)_{t,s \in Z^{c}},  \,\,\,   \wedge^2 (B_1 \brokenvert B_2) \bigg) 
$$
and
 $$ I_2 = \Bigg( \bigg(\sum_{a=1}^{n-r} x_{i,a}x_{d+1-j,d+1-a}  \bigg)_{i,j \in Z}, \,\,\, \wedge^2 (B_1 \brokenvert B_2) \Bigg).$$
    \subsubsection{} 
When $l=d-2$, the irreducible components of $\overline U_{d,l}$ are the closed subschemes $ V(I_1),V(I_2)$ and $V(I_3)$ where:
$$
I_1=  \Bigg( (x_{i,1})_{i\in Z} \Bigg), \quad I_2 =  \Bigg( (x_{i,d})_{i\in Z} \Bigg)
$$
and
$$ 
I_3 = \Bigg(  \sum_{a=n-(r-1)}^{n} x_{a,1}x_{d+1-a,1}+\frac{1}{2}x^2_{n+1,1},\quad   \sum_{a=n-(r-1)}^{n} x_{a,d}x_{d+1-a,d}+\frac{1}{2}x^2_{n+1,d},
$$ 
$$ 
 \wedge^2 (B_1 \brokenvert B_2), \quad  \sum_{a=n-(r-1)}^{n} x_{a,1}x_{d+1-a,d}+\frac{1}{2}x_{n+1,1}x_{n+1,d}   \Bigg) . 
$$
\smallskip

\subsection{}  $(d,l)=(\odd,\even)$. 

\subsubsection{} When $l>2$ the irreducible components of $\overline U_{d,l}$ are the closed subschemes $ V(I_1)$ and $V(I_2)$ where:
 $$ 
 I_1= \Bigg(  \bigg( \sum_{a= n-(r-1)}^{n} x_{d+1-a,d+1-t}x_{a,s}\bigg)_{t,s \in (Z')^{c}},\quad \wedge^2 ((B_1 \brokenvert Q \brokenvert B_2),
 $$
 $$
 \bigg(\sum_{a=n-(r-1)}^{n} x_{d+1-a,n+1}x_{a,j}\bigg)_{1\leq j \leq d}\Bigg)
 $$ 
and    
$$ 
I_2 = \Bigg(  \wedge^2 ((B_1 \brokenvert Q \brokenvert B_2), \quad \bigg(\sum_{a=1}^{n-r} x_{i,a}x_{d+1-j,d+1-a}+\frac{1}{2}x_{i,n+1}x_{d+1-j,n+1}   \bigg)_{i,j \in Z'}\Bigg). $$
    
    \subsubsection{} 
When $l=2$, the irreducible components of $\overline U_{d,l}$ are the closed subschemes $ V(I_1),V(I_2)$ and $V(I_3)$ where:
$$
I_1=  \Bigg( (x_{n,i})_{i\in (Z')^{c}} \Bigg), \quad I_2 =  \Bigg( (x_{n+2,i})_{i\in (Z')^{c}} \Bigg)
$$
and
$$ 
I_3 = \Bigg(  \sum_{a=1}^{n-1} x_{n,a}x_{n,d+1-a}+\frac{1}{2}x^2_{n,n+1},\quad  \sum_{a=1}^{n-1} x_{n+2,a}x_{n+2,d+1-a}+\frac{1}{2}x^2_{n+2,n+1},
$$ 
$$ 
\wedge^2 ((B_1 \brokenvert Q \brokenvert B_2), \quad \sum_{a=1}^{n-1} x_{n,a}x_{n+2,d+1-a}+\frac{1}{2}x_{n,n+1}x_{n+2,n+1}   \Bigg) . 
$$
\smallskip

\subsection{}  $(d,l)=(\even,\odd)$. The irreducible components of $\overline U_{d,l}$ are the closed subschemes $ V(I_1)$ and $V(I_2)$ where:
$$
I_1= \Bigg(   \bigg(\sum_{a=n-r}^{n-1} x_{d+1-a,n+1}x_{a,j}+\frac{1}{2}x_{n,n+1}x_{n,j}\bigg)_{1\leq j \leq d}, \quad \wedge^2 ((B_1 \brokenvert Q \brokenvert B_2),
$$
$$
\bigg( \sum_{a= n-r}^{n-1} x_{d+1-a,d+1-t}x_{a,s}+\frac{1}{2} x_{n,d+1-t}x_{n,s}\bigg)_{t,s \in (Z')^{c}} \Bigg), 
$$ 
$$ 
I_2 = \Bigg(   \bigg( \sum_{a=1}^{n-r-1} x_{i,a}x_{n,d+1-a} +\frac{1}{2}x_{i,n+1}x_{n,n+1}  \bigg)_{i \in Z'}, \quad
\wedge^2 ((B_1 \brokenvert Q \brokenvert B_2),
$$
$$
 \bigg( \sum_{a=1}^{n-r-1} x_{i,a}x_{d+1-j,d+1-a} +\frac{1}{2}x_{i,n+1}x_{d+1-j,n+1}  \bigg)_{i \in Z', j\in Z'\setminus \{ n \} }\Bigg).  
$$

\end{document}